\documentclass[a4paper,12pt,reqno]{amsart}

\usepackage[a4paper,hscale=0.7,vscale=0.7,centering]{geometry}

\usepackage[english]{babel}
\usepackage[psamsfonts]{amssymb}
\usepackage{cmmib57}
\usepackage{exscale}
\usepackage{amsmath}
\usepackage{amscd}
\usepackage{latexsym}
\usepackage{amsthm}
\usepackage{pdfsync}
\usepackage[T1]{fontenc}
\usepackage[latin1]{inputenc}
\usepackage[dvips]{graphicx}

\usepackage{color}
\definecolor{rp}{rgb}{0.25, 0, 0.75}
\definecolor{dg}{rgb}{0, 0.5, 0}

\newcommand{\hy}{\hat y}
\newcommand{\iou}{\int_{0}^{1}}

\newcommand{\ou}{[0,1]}
\newcommand{\1}{{\bf 1}}

\newcommand{\D}{\mathbb D}
\newcommand{\R}{\mathbb R}
\newcommand{\N}{\mathbb N}

\newcommand{\cb}{\mathcal B}
\newcommand{\cac}{\mathcal C}

\newcommand{\cd}{\mathcal D}
\newcommand{\ce}{\mathcal E}
\newcommand{\cf}{\mathcal F}
\newcommand{\ch}{\mathcal H}

\newcommand{\cl}{\mathcal L}

\newcommand{\cs}{\mathcal S}

\newcommand{\crr}{\mathcal R}
\newcommand{\ck}{\mathcal K}

\newcommand{\lp}{\left(}
\newcommand{\rp}{\right)}

\newcommand{\lln}{\left|}
\newcommand{\rrn}{\right|}

\newcommand{\ga}{\gamma}
\newcommand{\la}{\lambda}

\newcommand{\oom}{\Omega}

\newcommand{\si}{\sigma}

\newcommand{\vp}{\varphi}

\newcommand{\al}{\alpha}
\newcommand{\Om}{\Omega}
\newcommand{\om}{\omega}
\newcommand{\del}{\delta}

\newcommand{\ep}{\varepsilon}
\newcommand{\lam}{\lambda}

\newcommand{\Gam}{\Gamma}
\newcommand{\gam}{\gamma}
\newcommand{\ffi}{\varphi}
\newcommand{\hac}{\mathcal{H}}
\newcommand{\re}{\mathbb{R}}

\newcommand{\sig}{\sigma}
\newcommand{\ka}{\kappa}

\newcommand{\tf}{\mathcal{F}}
\newcommand{\beq}{\begin{equation}}
\newcommand{\eeq}{\end{equation}}

\newtheorem{theorem}{Theorem}[section]

\newtheorem{definition}[theorem]{Definition}

\newtheorem{lemma}[theorem]{Lemma}

\newtheorem{proposition}[theorem]{Proposition}

\theoremstyle{remark}
\newtheorem{remark}[theorem]{Remark}

\begin{document}

\title[Pathwise second order SDEs]{Pathwise definition of second order SDEs}

\author{Llu\'{\i}s Quer-Sardanyons \and Samy Tindel}

\address{
{\it Llu\'{\i}s Quer-Sardanyons:}
{\rm Departament de Matem\`atiques, Facultat de Ci\`encies, Edifici C, Universitat Aut\`onoma de Barcelona, 08193 Bellaterra, Spain}.
{\it Email: }{\tt quer@mat.uab.cat}
\newline
$\mbox{ }$\hspace{0.1cm}
{\it Samy Tindel:}
{\rm Institut \'Elie Cartan Nancy, B.P. 239,
54506 Vandoeuvre-l\`es-Nancy Cedex, France}.
{\it Email: }{\tt tindel@iecn.u-nancy.fr}
}



\date{\today}

\thanks{L. Quer-Sardanyons is supported by the grant MCI-FEDER Ref. MTM2009-08869 and S. Tindel is partially supported by the (French) ANR grant ECRU}

\begin{abstract}
In this article, a class of  second order differential equations on $\ou$, driven by a $\ga$-H\"older continuous function for any value of $\ga\in(0,1)$ and with multiplicative noise, is considered. We first show how to solve this equation in a pathwise manner, thanks to Young integration techniques. We then study the differentiability of the solution with respect to the driving process and consider the case where the equation is driven by a fractional Brownian motion, with two aims in mind: show that the solution we have produced coincides with the one which would be obtained with Malliavin calculus tools, and prove that the law of the solution
is absolutely continuous with respect to the Lebesgue measure.
\end{abstract}

\maketitle

\section{Introduction}
\label{sec:intro}

During the last past years, a growing activity has emerged, aiming at solving stochastic PDEs beyond the Brownian case. In some special situations, namely in linear (additive noise) or bilinear (noisy term of the form $u \, \dot B)$ cases, stochastic analysis techniques can be applied \cite{LS,TTV}. When the driving process of the equation exhibits a H\"older continuity exponent greater than $1/2$, Young integration or fractional calculus tools also allow to solve those equations in a satisfying way \cite{GLT,MN,QT}. Eventually, when one wishes to tackle non-linear problems in which the driving noise is only H\"older continuous with H\"older regularity exponent $\le 1/2$, rough paths analysis must come into the picture. This situation is addressed in \cite{DGT,GT,Te}.

\smallskip

It should be mentioned however that all the articles mentioned above only handle the case of  parabolic or hyperbolic systems, letting apart the case of elliptic equations. This is of course due to the special physical relevance of heat and wave equations, but also stems from a specific technical difficulty inherent to elliptic equations. Indeed, even in the usual Brownian case, the notion of filtration and adapted process is useless in order to solve non-linear elliptic systems, so that It\^o's integration theory is not sufficient in this situation. A natural idea in this context is then to use the power of anticipative calculus, based on Malliavin type techniques (see e.g. \cite{nualart}). This method has however a serious drawback in our context, mainly because the Picard type estimates involve Malliavin derivatives of any order, and cannot be closed. To the best of our knowledge, all the stochastic elliptic equations considered up to now involve thus a mere additive noise. Let us mention for instance the pioneering works \cite{bp,NP} for the existence and uniqueness of solutions, the study of Markov's property \cite{DP,NP}, the numerical approximations of \cite{MS,ST,Ti}, as well as the recent and deep contribution \cite{Pr}, which relates stochastic elliptic systems, anticipative Girsanov's transforms and deterministic methods.

\smallskip

With these preliminary considerations in mind, the aim of the current paper is twofold:

\smallskip

\noindent
\textbf{(i)} We wish to solve a nonlinear elliptic equation of the form 
\beq
\partial_{tt}^{2} z_t=\sig(z_t) \dot x_t,\quad t\in [0,1], \quad y_0=y_1=0,
\label{1}
\eeq
where $\si$ is a smooth enough function from $\R$ to $\R$, and $x$ is a  H\"older continuous noisy input with any H\"older continuity exponent $\ga\in(0,1)$. To this purpose, we shall write equation (\ref{1}) in a variant of the so-called mild form, under which it becomes obvious that the system can be solved in the space $\cac^{\ka}$ of $\ka$-H\"older continuous functions, for any $1-\ga<\ka<1$ (see Section \ref{prel} for a precise definition of this space). 

\smallskip

Let us observe however that, when dealing with a non-linear multiplicative noise, one is not allowed to use the monotonicity methods invoked in \cite{bp}. This forces us to use contraction type arguments, which can be applied only provided the H\"older norm of $x$ is small enough. In order to overcome this restriction, we shall introduce a positive constant $M$, and replace the diffusion coefficient $\si$ by a function $\si_M:\R\times\cac^\ga\to\R$ such that $y\mapsto\si_M(y,x)$ is regular enough and $\si_M(\cdot,x)\equiv 0$ whenever $\|x\|_{\ga}\ge M+1$. We shall thus produce a \emph{local} solution to equation (\ref{1}), in the sense given for instance in \cite{nualart} concerning the localization of the divergence operator on the Wiener space. Once this change is made, a proper definition of the solution plus a fixed point argument leads to the existence and uniqueness of solution for equation (\ref{1}).

\smallskip

\noindent
\textbf{(ii)} Having produced a unique solution to our system in a reasonable class of functions, one may wonder if this solution could have been obtained thanks to Malliavin calculus techniques, in spite of the fact that a direct application of those techniques to our equation do not yield a satisfying solution in terms of fixed point arguments. In order to answer this question, we shall prove that, when $x$ is a fractional Brownian motion (fBm in the sequel), the solution is differentiable enough in the Malliavin calculus sense, so that the stochastic integrals involved in the mild formulation of (\ref{1}) can be interpreted as Skorohod integrals plus a trace term, or better said as Stratonovich integrals. This will be achieved by differentiating the deterministic equation (\ref{1}) with respect to the driving noise $x$ and identifying this derivative with the usual Malliavin derivative, as done in \cite{BH,LT,NS}. As a by-product, we will also be able to study the density of the random variable $z_t$ for a fixed time $t\in(0,1)$.

\smallskip

We shall thus obtain the following result, which is stated here in a rather loose form (the reader is sent to the corresponding sections for detailed statements):
\begin{theorem}\label{thm:intro}
Consider $x\in\cac^\ga$ for a given $\ga>0$, a constant $M>0$ and a $\cac^4(\R)$ function $\si$, such that $\|\sig^{(j)}\|_{\infty}\leq \frac{c_j}{M+1}$ for any $j=0,1,2$ with some small enough constants $c_j$. Let $\si_M$ be the localized diffusion coefficient alluded to above (see  Definition \ref{hyp:phi} for more details). Then

\smallskip

\noindent
\emph{(1)} The equation
\begin{equation}\label{eq:elliptic-intro2}
\partial_{tt}^{2} z_t=\sig_M(x,z_t) \dot x_t,\quad t\in [0,1], \quad z_0=z_1=0
\end{equation}
admits a unique solution, lying in a space of the form $\cac^\ka$ for any $1-\ga<\ka<1$.

\smallskip

\noindent
\emph{(2)} Assume $x$ to be the realization of a fractional Brownian motion with Hurst parameter $H>1/2$. Then for any $t\in[0,1]$, $z_t$ is an element of the Malliavin-Sobolev space $\D^{1,2}$ and the integral form of (\ref{eq:elliptic-intro2}) can be interpreted by means of Skorohod integrals plus trace terms (see Section \ref{fbm} for further definitions).

\smallskip

\noindent
\emph{(3)}  Still in the fBm context, with a slight modification of our cutoff coefficient $\si_M$ and under the non-degeneracy condition $|\sig(y)|\geq \sig_0>0$ for all $y\in \re$, one gets the following result: for any $t\in(0,1)$ and $a>0$, the restriction of $\cl(z_t)$ to $\R\backslash (-a,a)$ admits a density with respect to Lebesgue's measure.
\end{theorem}
The reader might wonder why we have made the assumption of a \emph{small} coefficient $\si$ here, through the assumption $\|\sig^{(j)}\|_{\infty}\leq \frac{c_j}{M+1}$. This is due to the fact that monotonicity methods, which are essential in the deterministic literature (see e.g. \cite{Ev}) as well as in the stochastic references quoted above, are ruled out here by the presence of the diffusion coefficient in front of the noise $\dot x$. We have thus focused on contraction type properties, which are also mentioned in \cite{NP}. Let us also say a word about possible generalizations to elliptic equations in dimension $d=2,3$: the main additional difficulty lies in the fact that the fundamental solution to the elliptic equation exhibits some singularities on the diagonal, which should be dealt with. In particular, if one wishes to handle the case of a general Hölder continuous signal $x$, rough paths arguments in higher dimensions should be used. This possibility goes far beyond the current article.

\smallskip

At a technical level, let us mention that the first part of Theorem  \ref{thm:intro} above relies on an appropriate formulation of the equation, which enables to quantify the increments of the candidate solution in a reasonable way, plus some classical contraction arguments. As far as the Malliavin differentiability of the solution is concerned, it hinges on rather standard methods (see~\cite{LT,NS}). However, our density result for $\cl(y_t)$ is rather delicate, for two main reasons:
\begin{itemize}
\item 
The lack of a real time direction or filtration in equation (\ref{eq:elliptic-intro2}) makes many usual lower bounds on the Malliavin derivatives rather clumsy.
\item
One has to take care of the derivatives of our cutoff function $\si_M$ with respect to the driving process, for which upper bounds are to be provided and compared to some leading terms in the Malliavin derivatives.
\end{itemize}
Solutions to these additional problems are given at Section \ref{fbm}, which can be seen as the most demanding part of our paper. It should also be pointed out that we are able to solve equation (\ref{eq:elliptic-intro2}) for any H\"older regularity of the driving noise $x$, while our stochastic analysis part is devoted to fBm with Hurst parameter $H>1/2$. This is only due to the fact that Malliavin calculus is much easier to handle in the latter situation, and we firmly believe that our results could be generalized to $H<1/2$.

\smallskip

Here is how our article is structured: our equation is defined and solved at Section~\ref{sec:existence-uniqueness}. Differentiation properties of its solution with respect to the driving process are investigated at Section~\ref{differentiability}. Finally, the Malliavin calculus aspects for fractional Brownian motion, including the existence of a density, are handled at Section~\ref{fbm}.

\smallskip

Unless otherwise stated, any constant $c$ or $C$ appearing in our computations below is understood as a generic constant which might change from line to line without further mention.

\section{Existence and uniqueness of solution}\label{sec:existence-uniqueness}
Recall that we wish to solve the one-dimensional second order differential equation~(\ref{1}). Towards this aim, we shall change a little its formulation thanks to some heuristic considerations, and introduce our localization coefficient $\si_M$. We will then be able to solve the equation thanks to a fixed point argument.

\subsection{Heuristic considerations}
\label{prel}

Assume for the moment that $x$ is a smooth function defined on $\ou$. Hence, if $\si$ is small and regular enough, it is easily shown (see \cite{NP} for similar arguments) that equation~(\ref{1}) can be solved thanks to contraction arguments.

\smallskip

It is also well-known in this case that equation (\ref{1}) can be understood in the mild sense. Specifically, let the kernel $K:\ou^2\to\ou$ be the fundamental solution of the linear elliptic equation with Dirichlet boundary conditions, and notice that this kernel is explicitly given by
\begin{equation}\label{eq:def-kernel}
K(t,\xi)=t\wedge\xi-t\xi, \quad t,\xi\in\ou.
\end{equation}
Then $\{z_t,\; t\in [0,1]\}$ solves (\ref{1}) if it satisfies the
integral equation
\beq
z_t=\int_0^1 K(t,\xi)\sig (z_\xi) dx_\xi,\quad t\in [0,1],
\label{2}
\eeq
where the integrals above are understood in the Riemann sense as soon as $x$ is continuously differentiable.

\smallskip

Still assuming that $x$ is continuously differentiable, let us retrieve some more information about the increments of the solution $y$ to our elliptic equation. In order to do so, set first
\begin{equation*}
\del f_{st}=f_t-f_s, \qquad 0\le s\le t \le 1,
\end{equation*}
for any continuous function $f$. Let us also give an expression for the increments of $K$, by noticing that this kernel can be differentiated with respect to its first variable. Indeed, one has
$$
\partial_u K(u,\xi)=\1_{\{u\leq \xi\}} -\xi \quad \Longrightarrow\quad
K(t,\xi)-K(s,\xi)=\int_s^t \lp \1_{\{u\leq \xi\}} -\xi \rp \, du.
$$
Then, thanks to an obvious application of Fubini's theorem,  the increments of $z$ can be written as
\begin{align*}
(\del z)_{st} & = \int_0^1 \left( \int_s^t (\1_{\{u\leq \xi\}} -\xi) du\right) \sig(z_\xi) dx_\xi \\
& = \int_s^t du \int_0^1  (\1_{\{u\leq \xi\}} -\xi) \sig(z_\xi) dx_\xi\\
& =  \int_s^t du \left(\int_u^1 \sig(z_\xi) dx_\xi\right) - (t-s)\int_0^1 \xi \sig(z_\xi) dx_\xi.
\end{align*}

\smallskip

The latter equation is the one which is amenable to generalization to a non-smooth setting, and we will thus interpret our elliptic system in this way: we say that a continuous function $z:\ou\to\R$ is a solution to (\ref{1}) if, for any $0\le s\le t \le 1$,
\beq
\del z_{st} = \int_s^t du \left(\int_u^1 \sig(z_\xi) dx_\xi\right) - (t-s)\int_0^1 \xi \sig(z_\xi) dx_\xi,
\label{3}
\eeq
where the integrals with respect to the driving noise $x$ are interpreted in the Young sense.

\subsection{H\"older spaces and cutoff}
\label{sec:cutoff}
Though it could be intuited from the original equation, our formulation (\ref{3}) of the elliptic system indicates clearly that the candidate solution should be $\kappa$-H\"older continuous for
any $\kappa<1$, independently of the smoothness of $x$.

\smallskip

More precisely, let $\cac^\gam$ be the space of continuous functions $f\in \cac([0,1])$ such that $\|f\|_\gam<+\infty$, where
$$\|f\|_\gam^2=\|f\|_\infty + \sup_{0\le s < t \le 1}\frac{|\del f_{st}|}{|t-s|^\gam},$$
and where we recall that $\del f_{st}=f_t-f_s$. We shall define the integrals in (\ref{3}) thanks to the following classical proposition (see \cite{young})):
\begin{proposition}
Let $f\in \cac^\gam$, $g\in \cac^\kappa$ with $\gam+\kappa>1$, and $0\le s\le t\le 1$. Then the  integral $\int_s^t g_\xi df_\xi$ is well-defined as limit of Riemann sums along partitions of $[s,t]$. Moreover, the following estimation is fulfilled:
\beq
\left|\int_s^t g_\xi df_\xi\right| \leq c_{\ga,\ka} \|f\|_\gam \|g\|_\kappa |t-s|^\gam,
\label{y}
\eeq
where the constant $c_{\ga,\ka}$ only depends on $\gam$ and $\kappa$. A sharper estimate is also available:
\begin{equation}\label{eq:ineq-young-sharp}
\left|\int_s^t g_\xi df_\xi\right| \leq |g_s| \, \|f\|_\gam |t-s|^\gam
+ c_{\ga,\ka}  \|f\|_\gam \|g\|_\kappa |t-s|^{\gam+\ka}.
\end{equation}

\end{proposition}
The following straightforward property will also be used in the sequel: if $f,g\in \cac^\gam$, then the product $fg$ defines an element in $\cac^\gam$ such that
$\|fg\|_\gam\leq \|f\|_\gam \|g\|_\gam$.

\begin{remark}
It might be clear to the reader that the solution to our elliptic system will live in fact in a space of Lipschitz functions. We have chosen here to work in the Young setting because this does not induce any additional difficulty, and is more likely to be generalized to higher dimensions of the parameter $t$.
\end{remark}

\smallskip

The following Fubini type theorem for Young integrals is a slight modification of \cite[Proposition 2.6]{LT} and shall be needed in the sequel:
\begin{proposition}\label{prop:fubini-young}
Consider $\ga_i,\la_i \in(0,1)$, $i=1,2$, such that $\ga_i+\la_j>1$ for all $i,j=1,2$. Let $g\in \cac^{\ga_1}$ and $f\in \cac^{\ga_2}$, 
and $h:\{(t,s)\in [0,1]^2; 0\le s\le t\le 1\} \rightarrow\R$ a function such that $h(\cdot,t)$ (resp. $h(t,\cdot)$)
belongs to $\cac^{\la_1} ([t,1])$ (resp. $\cac^{\la_2}([0,t])$) uniformly in $t\in[0,1]$, and
\begin{equation}\label{eq:uhc}
\| h(r_1,\cdot)-h(r_2,\cdot)\|_{\la_2} \le C |r_1-r_2|^{\ga_1} \qquad \| h(\cdot,u_1)-h(\cdot,u_2)\|_{\la_1} \le C |u_1-u_2|^{\ga_2}.
\end{equation}
Then 
\begin{equation}\label{eq:ft}
\int_s^t\int_s^r h(r,u) \, dg_udf_r=
\int_s^t\int_u^t h(r,u) \, df_r dg_u,\quad 0\le s\le t\le T,
\end{equation}
and 
\begin{equation*}
\int_s^t\int_u^1 h(r,u) \, df_r dg_u =
\int_s^1 \int_s^{t\land r} h(r,u) \, dg_u df_r,\quad 0\le s\le t\le T.
\end{equation*}
\end{proposition}

\smallskip

Let us describe now our cutoff procedure on the coefficient $\si$. Recall that we wish to produce a smooth function $\si_M:\R\times\cac^\ga\to\R$ such that $\si_M(\cdot,x)\equiv 0$ whenever $\|x\|_{\ga}\ge M+1$. This also means that the H\"older norm of $x$ should enter into the picture in a smooth manner. To this purpose, let us consider the Sobolev type norm
\begin{equation*}
\|f\|_{\gam,p}:=
\left( \int_0^1\int_0^1 \frac{(f(\zeta)-f(\eta))^{2p}}{|\zeta-\eta|^{2p\gam+2}} \, d\zeta d\eta \right)^{\frac{1}{2p}},
\quad\mbox{for } p\ge 1.
\end{equation*}

It will be seen below that $\|f\|_{\gam,p}^{2p}$ can be differentiated with respect to $f$ in a suitable sense. Furthermore, Garsia's lemma (see e.g. \cite[Lemma 1]{Ga}) assesses that, whenever $2p \gam >1$, we have $\|f\|_\gam \leq C \|f\|_{\gam,p}$. Otherwise stated, we have the following:
\begin{remark}\label{fita}
Let $\gam\in (0,1)$. Assume that $\ep>0$ and $p\geq 1$ satisfy $\ep>\frac{1}{2p}$. Then:
$$f\in \cac^{\gam+\ep} \Longrightarrow \|f\|_{\gam,p}<\infty.$$
\end{remark}

\smallskip

This being said, our local coefficient is built in the following manner: let $M>0$ be an arbitrary strictly positive number. We introduce a smooth cutoff function $\ffi_M$ satisfying:
\begin{definition}\label{hyp:phi}
We consider a function $\ffi_M\in \cac^\infty_b((0,\infty))$ such that $\ffi_M(r)=0$
for all $r>M+1$, and $\ffi_M(r)=1$ for $r<M$.
For any $x:[0,1]\rightarrow \re$ for which $\|x\|_{\gam,p}<\infty$, for
some $\gam\in (0,1)$ and $p\geq 1$, set \beq G_M(x):=\ffi_M(\|x\|_{\gam,p}^{2p}). \label{35} \eeq Eventually, for such
$x$ and any $y\in \re$, we define \beq \sig_M(x,y):= G_M(x) \sig(y). \label{36} \eeq Hence, in particular,
$\sig_M(x,y)=0$ whenever $\|x\|_{\gam,p}^{2p} \geq M+1$. 
\end{definition}
We shall consider now the modified elliptic integral equation:
\beq
\del z_{st} = \int_s^t du \left(\int_u^1 \sig_M(x,z_\xi) dx_\xi\right) - (t-s)\int_0^1 \xi \sig_M(x,z_\xi) dx_\xi, \quad 0\le s\le t\le 1.
\label{4}
\eeq
That is, we will solve Equation (\ref{3}) for any control $x\in \cac^\gam$ such that $\|x\|_{\gam,p}^{2p}<M$. Notice in particular that the solution $z$ to (\ref{4}) depends on $M$, though we have avoided most of the explicit references to this fact for notational sake.

\subsection{Fixed point argument}
\label{existence}

After the preliminary considerations of Sections \ref{prel} and~\ref{sec:cutoff}, we now consider a driving signal $x$ in a H\"older space $\cac^\ga$, and we will  seek for a unique solution to equation (\ref{4}) in $\cac^{\ka}$ with $1-\ga<\ka<1$.

\smallskip

As it will be illustrated in the proof of Theorem \ref{t1}, we will need some regularity properties of $\sig$ when
considered as a map defined on $\cac^\kappa$ with values into itself. More precisely, we will make use of the following
result:

\begin{lemma}\label{l1}
Suppose that $\sig:\re\rightarrow \re$ is a bounded function that belongs to $\cac^2(\re)$ and has bounded derivatives. Then, for any $\kappa\in (0,1)$, $\sig:\cac^\kappa\rightarrow \cac^\kappa$ satisfies the following properties: for all $y,z\in \cac^\kappa$,
\[
\|\sig(y)\|_\kappa \leq \|\sig'\|_\infty \|y\|_\kappa + \|\sig\|_\infty,
\]
\[
\|\sig(y)-\sig(z)\|_\kappa \leq C \|y-z\|_\kappa \left\{ \|\sig'\|_\infty + \|\sig''\|_\infty (\|y\|_\kappa + \|y-z\|_\kappa)\right\}.
\]

\end{lemma}

\begin{proof}
The first part in the statement is an immediate consequence of the fact that $\sig$ and $\sig'$ are bounded
functions.

\smallskip

For the second part, let us fix $s,t\in [0,1]$ and $y,z\in \cac^\kappa$, so that we need to analyze the increment
$$\del(\sig(y)-\sig(z))_{st}= \sig(y_t)-\sig(z_t)-\sig(y_s)+\sig(z_s).$$
To this aim, let us consider the following path:
for any $\lam, \mu\in [0,1]$, set
$$a(\lam,\mu)=y_s + \lam(z_s-y_s) + \mu(y_t-y_s) +\lam \mu (y_s-y_t-z_s+z_t).$$
Notice that, in particular, $a(0,0)=y_s$, $a(0,1)=y_t$, $a(1,0)=z_s$ and $a(1,1)=z_t$. Then, we can write
\begin{align}
\del(\sig(y)-\sig(z))_{st} & = \int_0^1 d\lam \int_0^1 d\mu \, \partial_\lam \partial_\mu \sig(a(\lam,\mu))\nonumber \\
& = \int_0^1 d\lam \int_0^1 d\mu \left[ \sig'(a(\lam,\mu)) \partial_\lam \partial_\mu a (\lam,\mu) + \sig''(a(\lam,\mu)) \partial_\lam a(\lam,\mu)
\partial_\mu a(\lam,\mu)\right].
\label{5}
\end{align}
On the other hand, we have the following estimates:
$$|\partial_\lam \partial_\mu a (\lam,\mu)| \leq \|y-z\|_\kappa |t-s|^\kappa,$$
$$| \partial_\lam a(\lam,\mu) \partial_\mu a(\lam,\mu)| \leq C \|y-z\|_\kappa (\|y\|_\kappa + \|y-z\|_\kappa)|t-s|^\kappa.$$
Using these bounds and expression (\ref{5}), we end up with
$$|\del(\sig(y)-\sig(z))_{st}| \leq C \|y-z\|_\kappa \left\{ \|\sig'\|_\infty + \|\sig''\|_\infty (\|y\|_\kappa + \|y-z\|_\kappa)\right\} |t-s|^\kappa.$$
Therefore, we conclude the proof. 

\end{proof}

\smallskip

We are now in position to state the following existence and uniqueness result for Equation~(\ref{4}):

\begin{theorem}\label{t1}
Let $\gam,\kappa\in (0,1)$ be such that $\gam+\kappa>1$. Assume that $\ep>0$ and $p\geq 1$ satisfy $\ep>\frac{1}{2p}$
and let $x\in \cac^{\gam+\ep}$. Suppose that $\sig:\re\rightarrow \re$ is
bounded, belongs to $\cac^2(\re)$ and has bounded derivatives. Suppose also that the derivatives of $\sig$ satisfy the
following condition:
\begin{equation}
\|\sig^{(j)}\|_{\infty}\leq \frac{c_1}{M+1},\qquad j=0,1,2,
\label{sigma2}
\end{equation}
for a small enough constant $c_1<1$. Then, there exists a unique solution of Equation (\ref{4}) in $\cac^\kappa$.
Moreover, it holds that
\beq
\|z\|_\ka \leq C(M),
\label{eq:bound-z}
\eeq
where $C(M)$ is a positive constant depending on $M$.
\end{theorem}

\begin{proof}
As mentioned above, we will apply a fixed-point argument. Let us thus consider the following map on $\cac^\kappa$: for any $z\in \cac^\kappa$, $\Gam(z)$ is the element of $\cac([0,1])$ given by
$$\Gam(z)_t = \int_0^t du \left(\int_u^1 \sig_M(x,z_\xi) dx_\xi\right) - t \int_0^1 \xi \, \sig_M(x,z_\xi) dx_\xi,\quad t\in [0,1].$$
Owing to Lemma \ref{l1} and the definition of $\sig_M$, one easily proves that, for all $z\in \cac^\kappa$, $\Gam(z)$ is well-defined and belongs to $\cac^\kappa$.
We aim to prove that  $\Gam:\cac^\kappa\rightarrow \cac^\kappa$  has a unique fixed point. For this, we will find an invariant ball in $\cac^\kappa$ under $\Gam$ and check that $\Gam$, restricted to that ball, defines a contraction.

\smallskip

To begin with, let us fix a real number $K>1$ and consider the following closed ball in the H\"older space $\cac^\kappa$:
$$\cb_K :=\{z\in \cac^\kappa,\; \|z\|_\kappa \leq K\}.$$
Next, for $z\in \cb_K$, we are going to analyze the norm $\|\Gam(z)\|_\kappa$. Indeed, for any $s,t\in [0,1]$, $s<t$, we have that
\begin{align*}
|\del(\Gam(z))_{st}| & \leq \int_s^t du \left| \int_u^1 \sig_M(x,z_\xi) dx_\xi \right| + |t-s| \left| \int_0^1 \xi \, \sig_M(x,z_\xi) dx_\xi\right| \\
& \leq C_1 G_M(x) \|x\|_\gam \|\sig(z)\|_\kappa |t-s| \\
& \leq C_1 G_M(x) \|x\|_\gam (\|\sig'\|_\infty \|z\|_\kappa + \|\sig\|_\infty) |t-s|,
\end{align*}
where in the last inequality we have applied Lemma \ref{l1} and $C_1$ denotes a positive constant. Furthermore, 
the above estimate let us also infer that
\[
 \|\Gam(z)\|_\infty \leq C_1 G_M(x) \|x\|_\gam (\|\sig'\|_\infty \|z\|_\kappa + \|\sig\|_\infty).
\]
Hence,
\beq
\|\Gam(z)\|_\kappa \leq C_1 G_M(x) \|x\|_\gam (\|\sig'\|_\infty \|z\|_\kappa + \|\sig\|_\infty).
\label{eq:exis-gam}
\eeq
Since $z\in \cb_K$ and $G_M(x) \|x\|_\gam< M$, we get
$$\|\Gam(z)\|_\kappa  \leq C_1 M ( K \|\sig'\|_\infty + \|\sig\|_\infty).
\le C_1 \, c_1 (K+1),
$$
thanks to (\ref{sigma2}).
Moreover, recall that we have chosen a constant $K>1$. Therefore, by the hypothesis on $\sig$, if we take for instance $c_1<(2 C_1)^{-1}$, we obtain $\|\Gam(z)\|_\kappa  \leq (K+1)/2$, and thus
$$\|\Gam(z)\|_\kappa \leq K \quad \text{whenever} \quad \|z\|_\kappa \leq K.
$$
This implies that $\cb_K$ is invariant under $\Gam$.

\smallskip

Let us now prove that $\Gam_{|_{\cb_K}} :\cb_K\rightarrow \cb_K$
is a contraction. For this, it suffices to show that $\Gam_{|_{\cb_K}}$ is Lipschitz with a Lipschitz constant smaller than $1$. Namely, we shall prove the existence of a constant $L<1$ such that, for all $y,z\in \cb_K$,
$$\|\Gam(y)-\Gam(z)\|_\kappa \leq L \|y-z\|_\kappa.$$
Let $s,t\in [0,1]$, $s<t$, and $y,z\in \cb_K$. Then,
\begin{align}
\del(\Gam(y)-\Gam(z))_{st} & = \int_s^t du \left( \int_u^1 [\sig_M(x,y_\xi)-\sig_M(x,z_\xi)] dx_\xi \right)  \nonumber \\
&\qquad -(t-s) \int_0^1 \xi \, [\sig_M(x,y_\xi)-\sig_M(x,z_\xi)] dx_\xi.
\label{6}
\end{align}
By Lemma \ref{l1} and the properties of the Young integral, it turns out that the absolute value of both terms on the right-hand side of (\ref{6})
can be bounded, up to some positive constant, by
$$G_M(x) \|x\|_\gam \|y-z\|_\kappa \left\{ \|\sig'\|_\infty + \|\sig''\|_\infty(\|y\|_\kappa + \|y-z\|_\kappa) \right\} |t-s|.$$
We have a similar bound for $\|\Gam(y)-\Gam(z)\|_\infty$ as well. Thus, because $y,z\in \cb_K$, we eventually end up with
$$\|\Gam(y)-\Gam(z)\|_\kappa \leq C_2 M K  ( \|\sig'\|_\infty + \|\sig''\|_\infty ) \|y-z\|_\kappa.$$
It suffices now to consider that $\|\sig'\|_\infty$ and $\|\sig''\|_\infty$ are sufficiently small (that is
we can take $c_1<(C_2 K+1)^{-1}\land (2 C_1)^{-1}$, where $C_1$ is the constant of the first part of the proof)
so that the right-hand side above is bounded by $L  \|y-z\|_\kappa$, with $L<1$.
Therefore, $\Gam$ has a unique fixed point in $\cb_K$, which means that  Equation
(\ref{4}) has a unique solution in $\cac^\kappa$. 

\smallskip

Eventually, using (\ref{eq:exis-gam}) one proves that
\[
 \|z\|_\kappa \leq C_1 M (\|\sig'\|_\infty \|z\|_\kappa + \|\sig\|_\infty).
\]
In addition, invoking (\ref{sigma2}) and the fact that $c_1\leq (2C_1)^{-1}$, we obtain
\[
 \|z\|_\kappa\leq \frac{C_1 M \|\sig\|_\infty}{1-C_1 M \|\sig'\|_\infty}\leq C(M),
\]
which concludes the proof.

\end{proof}

\begin{remark}\label{rmk:compact-eq}
Having been able to solve equation (\ref{4}) in $\cac^\kappa$ for any $\ka<1$, one can now apply the Fubini type Proposition \ref{prop:fubini-young} in order to assess that $z$ is the unique solution to the integral equation
\begin{equation*}
z_t=\iou K(t,\xi) \si_M(x,z_\xi) \, dx_\xi, \quad t\in\ou,
\end{equation*}
where we recall that the kernel $K(t,\xi)$ is defined by $K(t,\xi)=t\wedge\xi-t\xi$.
\end{remark}

\section{Differentiability of the solution with respect to the control}
\label{differentiability}

This section is devoted to show that the solution of Equation (\ref{4}) is differentiable, in the sense of Fr\'echet, when considered as a function of the control $x$
driving the equation. For this, we need two auxiliary results.

\smallskip

Let us remind that the diffusion coefficient under consideration (see Equation (\ref{4})) is  introduced in our Definition \ref{hyp:phi}. Furthermore, the following differentiation rule holds true:
\begin{proposition}\label{prop1}
Let $\gam,\kappa\in (0,1)$ be such that $\gam+\kappa>1$, $p\geq 1$ and $\ep>\frac{1}{2p}$. Assume that $\sig\in
\cac^4(\re)$ is bounded together with all its derivatives and let $\sig_M$ be given by Definition~\ref{hyp:phi}. Consider $x$
an element of $\cac^{\gam+\ep}$ and define the following map:
$$F:\cac^{\gam+\ep}\times \cac^\kappa\longrightarrow \cac^\kappa,$$
where, for all $h\in\cac^{\gam+\ep}$ and $z\in \cac^\kappa$,
$$F(h,z)_t:=z_t-\int_0^t du \left( \int_u^1 \sig_M(x+h,z_\xi) d(x+h)_\xi\right)
+t \int_0^1 \xi \, \sig_M(x+h,z_\xi) d(x+h)_\xi.$$

Then, the map $F$ is Fr\'echet differentiable with respect to the first and second variable
and the Fr\'echet derivatives are given by, respectively: for all $t\in [0,1]$, $k\in
\cac^{\gam+\ep}$ and $g\in \cac^\kappa$,
\begin{align}\label{7}
&(D_1F(h,z)\cdot k)_t \\
&=   -\int_0^t du \left[ \int_u^1 \sig_M(x+h,z_\xi)\, dk_\xi +
\int_u^1(D G_M(x+h)\cdot
k) \,\sig(z_\xi)\, d(x+h)_\xi\right]  \notag \\
& + t \left[\int_0^1 \xi\; \sig_M(x+h,z_\xi)\, dk_\xi + \int_0^1 \xi\, (D G_M(x+h)\cdot k)\, \sig(z_\xi)\, d(x+h)_\xi\right], \notag
\end{align}
and
\begin{multline}
(D_2F(h,z)\cdot g)_t =  \; g_t -\int_0^t du \left[ \int_u^1 G_M(x+h)\, \sig'(z_\xi)\,
g_\xi\, d(x+h)_\xi \right]  \\
+ t \int_0^1 \xi\, G_M(x+h)\, \sig'(z_\xi) \,g_\xi \,d(x+h)_\xi \, .
\label{8}
\end{multline}
\end{proposition}


\begin{remark}\label{remark1}
In the above formulae (\ref{7}) and (\ref{8}), the Fr\'echet derivative of $G_M(\cdot)$ is well-defined and can be computed explicitly. Indeed, $G_M$ is
defined on the H\"older space $\cac^{\gam+\ep}$, takes values in $\re$ and is defined by $G_M(x)=
\ffi_M(\|x\|_{\gam,p}^{2p})$, with some $p\geq 1$. Moreover, $\ffi_M$ is a smooth function
which fulfills Hypothesis \ref{hyp:phi}. Hence, the Fr\'echet
derivative $D G_M(x)$ at any point $x\in \cac^{\gam+\ep}$ defines a linear map on
$\cac^{\gam+\ep}$ with values in $\re$, and it is straightforward to check that it is
given by
$$DG_M(x)\cdot k = 2p \, \ffi_M'(\|x\|^{2p}_{\gam,p}) \int_0^1\int_0^1
\frac{(x_\zeta-x_\eta)^{2p-1} (k_\zeta-k_\eta)}{|\zeta-\eta|^{2\gam p +2}} d\zeta d\eta,\quad k\in \cac^{\gam+\ep}.
$$
Moreover, we have that
\begin{equation}\label{eq:ineq-DG-M}
\|D G_M(x)\|:=\|D G_M(x)\|_{\cl(\cac^{\gam+\ep};\, \R)}\leq C_p \|x\|_{\gam+\ep},
\end{equation}
where the norm on the left-hand side denotes the corresponding operator norm.
\end{remark}

\begin{remark}\label{rmk:compact-deriv}
As in Remark \ref{rmk:compact-eq}, one can apply Fubini's theorem for Young integrals in order to obtain some more compact expressions for the derivatives of $F$. Indeed, it is readily checked that
\begin{multline}
\label{eq:compact-deriv1}
(D_1F(h,z)\cdot k)_t =
-G_M(x+h) \iou K(t,\xi) \si(z_\xi) \, dk_\xi  \\
- (D G_M(x+h)\cdot k) \iou K(t,\xi) \si(z_\xi) \, d(x+h)_\xi
\end{multline}
and
\begin{equation*}
(D_2F(h,z)\cdot g)_t =   g_t - G_M(x+h) \iou K(t,\xi) \si'(z_\xi) g_\xi \,d(x+h)_\xi,
\end{equation*}
where $K$ is the kernel defined by (\ref{eq:def-kernel}).
\end{remark}

\begin{remark}\label{remark2}
As it will be explained later on in the paper, we will apply the results of this section to the case where $x$ is a fractional Brownian motion with Hurst parameter $H>\frac12$, defined on a complete probability space $(\Omega,\cf,P)$. In particular, the paths of $x$ are almost surely $\gam$-H\"older
continuous for all $\gam<H$, with $\gam$-H\"older norm in $L^p(\oom)$ for any $p\ge 1$. Thus, if we fix $\gam<H$, we will be able to find $\ep>1/(2p)$ satisfying $\gam+\ep<H$. This opens the possibility to apply the results of the current section to this particular case.
\end{remark}


\begin{proof}[Proof of Proposition \ref{prop1}]

Though the following considerations might be mostly standard (see \cite{LT,NS} for similar calculations), we include most of the details here for the sake of clarity. We will develop the proof in several steps.

\smallskip

\noindent
{\it Step 1.} First of all, let us prove that $F$ is
continuous. For this, let $h,\tilde h\in \cac^{\gam+\ep}$ and $z,\tilde z\in
\cac^\kappa$, so that we need to study the increment $\del(F(h,z)-F(\tilde h,\tilde
z))_{st}$, for $0\leq s<t\leq 1$. Indeed, we have that
\beq
|\del(F(h,z)-F(\tilde h,\tilde z))_{st}|\leq A_1+A_2+A_3,
\label{14}
\eeq
where
$$A_1=|\del(z-\tilde z)_{st}|,$$
$$A_2= \int_s^t du \left| \int_u^1\sig_M(x+h,z_\xi)\, d(x+h)_\xi -
\int_u^1\sig_M(x+\tilde h,\tilde z_\xi)\, d(x+\tilde h)_\xi\right|,$$
$$A_3= (t-s) \left| \int_0^1 \xi\, \sig_M(x+h,z_\xi)\, d(x+h)_\xi -
\int_0^1 \xi\, \sig_M(x+\tilde h,\tilde z_\xi)\, d(x+\tilde h)_\xi \right|.$$ It is clear
that
$$A_1\leq \|z-\tilde z\|_\kappa (t-s)^\kappa.$$
On the other hand, the term $A_2$ can be decomposed as $A_2\le A_{11}+A_{12}$, with:
\begin{eqnarray}\label{9}
A_{11}&=& \int_s^t du \left| \int_u^1 \left(\sig_M(x+h,z_\xi) - \sig_M(x+\tilde h,\tilde
z_\xi)\right)\, d(x+h)_\xi\right| \notag\\
A_{12}&=&\int_s^t du \left| \int_u^1 \sig_M(x+\tilde h,\tilde z_\xi) \, d(h-\tilde
h)_\xi\right|  
\end{eqnarray}
In addition, our bound (\ref{y}) on Young type integrals easily yields
\beq
A_{12}\leq G_M(x+\tilde h) (\|\sig'\|_\infty \|\tilde z\|_\kappa + \|\sig\|_\infty) \|h-\tilde h\|_{\gam + \ep} (t-s).
\label{10}
\eeq
We still need to bound the term $A_{11}$ by a sum $B_1+B_2$, where the latter terms are
defined by:
$$B_1= \int_s^t du \left| \int_u^1 \left(\sig_M(x+h,z_\xi) - \sig_M(x+ h,\tilde
z_\xi)\right)\, d(x+h)_\xi\right|,$$
$$B_2= \int_s^t du \left| \int_u^1 \left(\sig_M(x+h,\tilde z_\xi) - \sig_M(x+\tilde h,\tilde
z_\xi)\right)\, d(x+h)_\xi\right|.$$
Now, invoking  Lemma \ref{l1}, we get
\begin{align}
B_1 & \leq  C G_M(x+h) \|x+h\|_{\gam+\ep}  \|\sig(z)-\sig(\tilde z)\|_\kappa  (t-s)\nonumber\\
&  \leq C G_M(x+h) \|x+h\|_{\gam+\ep} \left(\|\sig'\|_\infty + \|\sig''\|_\infty
(\|z\|_\kappa + \|z-\tilde z\|_\kappa)\right) \|z-\tilde z\|_\kappa (t-s). \label{11}
\end{align}
Concerning the term $B_2$, notice that we clearly have
\begin{align*}
B_2 &\leq  \left|G_M(x+h)-G_M(x+\tilde h)\right| \int_s^t \left|\int_u^1 \sig(\tilde
z_\xi)\, d(x+h)_\xi\right| du \\
& \leq \left|G_M(x+h)-G_M(x+\tilde h)\right| \|x+h\|_{\gam+\ep} (\|\sig'\|_\infty \|\tilde z\|_\kappa + \|\sig\|_\infty) (t-s).
\end{align*}
Let us eventually analyse the difference $|G_M(x+h)-G_M(x+\tilde h)|$ on the right hand-side above: by
definition of $G_M$ and the properties of $\ffi_M$ summarized in Hypothesis \ref{hyp:phi}, we can argue as follows:
\begin{align}
|G_M(x+h)-G_M(x+\tilde h)| & = |\ffi_M(\|x+h\|_{\gam,p}^{2p})-\ffi_M(\|x+\tilde h\|_{\gam,p}^{2p})| \nonumber \\
&\leq C_{M,p} |\|x+h\|_{\gam,p}-\|x+\tilde h\|_{\gam,p}| \nonumber \\
& \leq C_{M,p} \|h-\tilde h\|_{\gam,p}
\label{17}
\end{align}
and this last term may be bounded, up to some constant, by $\|h-\tilde h\|_{\gam+\ep}$, because we have chosen $\ep$ to
be small but verifying $\ep>\frac{1}{2p}$ (see Remark \ref{fita}) . This implies that
\beq
B_2 \leq C \|h-\tilde
h\|_{\gam+\ep} \|x+h\|_{\gam+\ep} (\|\sig'\|_\infty \|\tilde z\|_\kappa + \|\sig\|_\infty) (t-s).
\label{12}
\eeq
Plugging the bounds
(\ref{10}), (\ref{11}) and (\ref{12}) in (\ref{9}), we obtain that
\begin{align}
A_2 & \leq   C_1 (\|\sig'\|_\infty \|\tilde z\|_\kappa + \|\sig\|_\infty) ( G_M(x+\tilde h) + \|x+h\|_{\gam+\ep})
\|h-\tilde h\|_{\gam+\ep} (t-s) \nonumber \\
& + C_2 G_M(x+h) \|x+h\|_{\gam+\ep} \left(\|\sig'\|_\infty + \|\sig''\|_\infty
(\|z\|_\kappa + \|z-\tilde z\|_\kappa)\right) \|z-\tilde z\|_\kappa (t-s),
\label{13}
\end{align}
where $C_1,C_2$ denote some positive constants.

\smallskip

The analysis for the term $A_3$ is very similar to that of $A_2$ and, indeed, for the former we end up with a similar bound as in (\ref{13}). Therefore,
going back to expression (\ref{14}), we have proved that
$$\|F(h,z)-F(\tilde h,\tilde z)\|_\kappa \leq C(M,\sig,x,z,\tilde z,h,\tilde h) (\|z-\tilde z\|_\kappa + \|h-\tilde h\|_{\gam+\ep}),$$
which implies that $F$ is continuous.

\smallskip

\noindent
{\it Step 2.} Let us prove now that the Fr\'echet derivative of $F$ with respect to $h$ is given by~(\ref{7}). First of all, let us check that
$D_1F(h,z):\cac^{\gam+\ep} \rightarrow \cac^\kappa$, as defined by expression~(\ref{7}), is a continuous map. Indeed, owing to inequality (\ref{y}) and Remark \ref{remark1}, one can easily check from expressions (\ref{7}) and (\ref{eq:ineq-DG-M}) that
$$\|D_1F(h,z)\cdot k\|_\kappa \leq C (\|\sig'\|_\infty \|z\|_\kappa + \|\sig\|_\infty) (G_M(x+h)+ \|x+h\|_{\gam+\ep}^2) \|k\|_{\gam+\ep},
$$
which implies that $D_1F(h,z)$ is continuous.

\smallskip

In order to prove that (\ref{7}) also represents the Fr\'echet derivative of $F$ with respect to the first variable, we fix $h\in
\cac^{\gam+\ep}$ and $z\in \cac^\kappa$, so that we need to prove that \beq \lim_{\|k\|_{\gam+\ep}\rightarrow 0}
\frac{\|F(h+k,z)-F(h,z)-D_1 F(h,z)\cdot k\|_\kappa}{\|k\|_{\gam+\ep}}=0. \label{19} \eeq For this, let $0\leq s<t\leq
1$, $h,k\in \cac^{\gam+\ep}$ and $z\in \cac^\kappa$, and we proceed to analyze the increment
\beq
|\del(F(h+k,z)-F(h,z)-D_1 F(h,z)\cdot k)_{st}|.
\label{15}
\eeq
According to (\ref{7}), the above increment can be
split into a sum of four terms, which we denote by $E_i$, $i=1,\dots,4$, and are defined as follows:
\begin{eqnarray*}
E_1&=&\int_s^t du \int_u^1 \left[ G_M(x+h)- G_M(x+h+k)- D G_M(x+h)\cdot k\right] \sig(z_\xi) \, d(x+h)_\xi, \\
E_2&=& \int_s^t du \int_u^1 \left[ G_M(x+h)- G_M(x+h+k) \right] \sig(z_\xi) \, dk_\xi, \\
E_3&=&(t-s) \int_0^1 \left[ G_M(x+h)- G_M(x+h+k)- D G_M(x+h)\cdot k\right] \xi\, \sig(z_\xi) \, d(x+h)_\xi, \\
E_4&=& (t-s) \int_0^1 \left[ G_M(x+h)- G_M(x+h+k) \right] \xi \,\sig(z_\xi) \, dk_\xi. \\
\end{eqnarray*}
We will only deal with the study of the terms $E_1$ and $E_2$, since the remaining ones involve analogous arguments. First, note that we have the following estimates:
\begin{align}
|E_1| & \leq  \left| G_M(x+h)- G_M(x+h+k)- D G_M(x+h)\cdot k\right|\, \left| \int_s^t du \int_u^1 \sig(z_\xi) \, d(x+h)_\xi \right| \nonumber \\
& \leq C \left| G_M(x+h)- G_M(x+h+k)- D G_M(x+h)\cdot k\right| \nonumber \\
& \hspace{6cm}  \times \|x+h\|_{\gam+\ep} (\|\sig'\|_\infty \|z\|_\kappa + \|\sig\|_\infty) (t-s).
\label{16}
\end{align}
By Remark \ref{remark1}, the map $G_M: \cac^{\gam+\ep}\rightarrow \re$ is Fr\'echet differentiable and its derivative can be computed explicitly.
Hence,
$$\lim_{\|k\|_{\gam+\ep}\rightarrow 0} \frac{\left| G_M(x+h)- G_M(x+h+k)- D G_M(x+h)\cdot k\right|}{\|k\|_{\gam+\ep}}=0,$$
and this implies that the contribution of $|E_1|$ is of order $o(\|k\|_{\gam+\ep})$.

\smallskip

On the other hand, using the same arguments as in (\ref{17}), we have:
\begin{align}
|E_2|& \leq C \left| G_M(x+h)-G_M(x+h+k) \right| \|k\|_{\gam+\ep} (\|\sig'\|_\infty \|z\|_\kappa + \|\sig\|_\infty) (t-s) \nonumber \\
& \leq C \|k\|_{\gam+\ep}^2 (\|\sig'\|_\infty \|z\|_\kappa + \|\sig\|_\infty) (t-s),
\label{18}
\end{align}
which is obviously also of order $o(\|k\|_{\gam+\ep})$.

\smallskip

For the terms $|E_3|$ and $|E_4|$ we obtain, respectively, the same bounds as in (\ref{16}) and~(\ref{18}). Eventually, plugging all these estimates in (\ref{15}), we end up with the limit (\ref{19}).

\smallskip

\noindent
{\it Step 3.} In this part, we prove that the Fr\'echet derivative of $F$ with respect to the second variable is given by (\ref{8}).
The continuity of $D_2F(h,z)$ in (\ref{8}) can be proved as we have done in Step 2 for $D_1F(h,z)$.  Hence,
we will check that, for all $h\in \cac^{\gam+\ep}$ and $z\in \cac^\kappa$, it holds:
\beq
\lim_{\|g\|_\kappa \rightarrow 0} \frac{\|F(h,z+g)-F(h,z)-D_2 F(h,z)\cdot g\|_\kappa}{\|g\|_\kappa}=0.
\label{20}
\eeq
Throughout this step we will use the fact that $\sig$, considered as a map defined on and taking values into $\cac^\kappa$, is Fr\'echet differentiable
and its derivative is given by (see Lemma \ref{l2} below):
$$(D \sig(z) \cdot g)_t = \sig'(z_t) g_t,\quad z,g\in \cac^\kappa.$$
This means that, for all $z\in \cac^\kappa$,
$$\lim_{\|g\|_\kappa \rightarrow 0} \frac{\|\sig(z+g)-\sig(z)-D\sig(z)\cdot g\|_\kappa}{\|g\|_\kappa}=0.$$
In order to prove (\ref{20}), let us fix $0\leq s<t\leq 1$ and observe that
$$\left| \del \left( F(h,z+g)-F(h,z)-D_2 F(h,z)\cdot g \right)_{st} \right| \leq F_1 + F_2,$$
where
\begin{align*}
F_1 & := \left| \int_s^t \int_u^1 G_M(x+h) [\sig(z_\xi+g_\xi)-\sig(z_\xi) - \sig'(z_\xi) g_\xi] d(x+h)_\xi\, du \right| \nonumber \\
& \leq C G_M(x+h) \|x+h\|_{\gam+\ep} \|\sig(z+g)-\sig(z)-\sig'(z) g\|_\kappa (t-s)
\end{align*}
and
\begin{equation*}
F_2  := (t-s) \left| \int_0^1 G_M(x+h) \xi \,[\sig(z_\xi+g_\xi)-\sig(z_\xi) - \sig'(z_\xi) g_\xi] d(x+h)_\xi \right|,
\end{equation*}
for which the same inequality as for $F_1$ is available. Therefore, we obtain that
\begin{align*}
& \|F(h,z+g)-F(h,z)-D_2 F(h,z)\cdot g\|_\kappa \\
& \qquad \qquad \leq C G_M(x+h) \|x+h\|_{\gam+\ep} \|\sig(z+g)-\sig(z)-\sig'(z) g\|_\kappa,
\end{align*}
and the latter $\kappa$-norm, as we have mentioned above, is of order $o(\|g\|_\kappa)$ whenever $\|g\|_\kappa$ tends to zero.
This implies that (\ref{20}) holds, and ends the proof. 

\end{proof}

\smallskip

Let us quote now the relation needed in the previous proof in order to compute the Fr\'echet derivative of the process $\si(z)$:
\begin{lemma}\label{l2}
Let $\sig\in \cac^4(\re)$ be a bounded function with bounded derivatives. Then $\sig$, understood as a map $\sig:\cac^\kappa\rightarrow \cac^\kappa$, is Fr\'echet differentiable and its derivative is given by:
$$(D \sig(z) \cdot g)_t = \sig'(z_t) g_t,\quad z,g\in \cac^\kappa.$$
\end{lemma}

\begin{proof}  
We refer to \cite[Proposition 3.5]{LT} for the proof of this fact, and in particular for the identification of $(D \sig(z) \cdot g)_t$ with the quantity $\si'(z_t)\cdot g_t$. 

\end{proof}

\smallskip

As in \cite{NS}, a crucial step in order to differentiate $z$ with respect to the driving noise $x$ is to solve the following class of linear elliptic PDEs:
\begin{proposition}\label{prop2}
Let $\gam, \kappa \in (0,1)$ be such that $\gam+\kappa>1$. Assume that we are given $x\in \cac^\gam$ and $w, R\in \cac^\kappa$ such that
the $\kappa$-norm of $R$ verifies
\beq
\|R\|_\kappa < \frac{c_2}{M+1},
\label{b-norm}
\eeq
for some small enough constant $c_2<1$.
Then, there exists a unique solution $\{y_t,\; t\in [0,1]\}$ in $\cac^\kappa$ of the following
linear integral equation: 
\begin{equation*}
y_t=w_t-G_M(x) \iou K(t,\xi) R_\xi \, y_\xi \, dx_\xi,
\quad t\in\ou.
\end{equation*}
Moreover, there exists a positive constant $c(M)$ that only depends on $M$ such that
\beq
\|y\|_\kappa \leq c(M) \|w\|_\kappa.
\label{22}
\eeq
\end{proposition}

\begin{proof} 
As for Theorem \ref{t1}, we will use a fixed point argument, and solve our equation under the form
\beq
(\del y)_{st} =  (\del w)_{st} -G_M(x) \int_s^t du \left( \int_u^1  R_\xi \, y_\xi dx_\xi\right)
+ (t-s) G_M(x) \int_0^1 \xi \,  R_\xi \, y_\xi dx_\xi,
\label{21}
\eeq
for any $0\leq s<t\leq 1$. More precisely, let us define the map $\Theta: \cac^\kappa\rightarrow \cac^\kappa$ by
\[
\Theta(y)_t := w_t -G_M(x) \int_0^t du \left( \int_u^1  R_\xi \, y_\xi dx_\xi\right)
+ t \, G_M(x) \int_0^1 \xi \, R_\xi \, y_\xi dx_\xi,
\]
for any $y\in \cac^\kappa$ and $t\in [0,1]$. Using elementary properties of Young integrals, one easily checks that the map $\Theta$ is well-defined,
that is $\Theta(y)$ belongs to $\cac^\kappa$ whenever $y\in \cac^\kappa$.

\smallskip

On the other hand, in order to prove that $\Theta$ defines a contraction, we will show that it exhibits a Lipschitz property with Lipschitz constant $L<1$.
Indeed, let us fix $y,\tilde y\in \cac^\kappa$ and proceed to study the increment $\del(\Theta(y)-\Theta(\tilde y))_{st}$ for any $0\leq s<t\leq 1$:
\begin{align*}
& |\del(\Theta(y)-\Theta(\tilde y))_{st}| \\
&\qquad  \leq  G_M(x) \int_s^t du \left| \int_u^1 R_\xi \, (y_\xi-\tilde y_\xi) \, dx_\xi \right|
+(t-s)\, G_M(x) \left| \int_0^1 \xi \, R_\xi \, (y_\xi-\tilde y_\xi) \, dx_\xi\right| \\
& \qquad  \leq C G_M(x) \|x\|_\gam \|R\|_\kappa \|y-\tilde y\|_\kappa (t-s).
\end{align*}
Hence, taking into account that $G_M(x) \|x\|_\gam < M$ and the assumptions on $R$, we conclude that
\[
\|\Theta(y)-\Theta(\tilde y)\|_\kappa \leq  c_2 \frac{C M}{M+1} \|y-\tilde y\|_\kappa.
\]
Choosing the constant $c_2$ conveniently, we get that $\Theta$ is Lipschitz with Lipschitz constant $L<1$. Therefore,
$\Theta$ defines a contraction and it has a unique fixed point, which solves equation (\ref{21}).

\smallskip

Eventually, the bound (\ref{22}) can be easily obtained using similar arguments as the ones developed so far. Indeed, 
observe that we have
\[
 |(\del y)_{st}|\leq C M \|R\|_\kappa \|y\|_\kappa (t-s) + \|w\|_\kappa (t-s)^\kappa
\]
and also $\|y\|_\infty \leq C M \|R\|_\kappa \|y\|_\kappa + \|w\|_\infty$. Thus
\[
 \|y\|_\kappa \leq C M \|R\|_\kappa \|y\|_\kappa + \|w\|_\kappa,
\]
from which one deduces (\ref{22}), provided our constant $c_2$ is chosen small enough.

\end{proof}

\smallskip

At this point, we can proceed to state and prove the main result of the section.

\begin{theorem}\label{t2}
Let $\gam, \kappa\in (0,1)$ be such that $\gam+\kappa>1$. Let $\ep>0$ and a sufficiently large $p\geq 1$ so that $\ep>\frac{1}{2p}$.
Assume that $\sig \in \cac^4(\re)$ is a bounded function with bounded derivatives such that:
\begin{equation}\label{eq:hyp-sigma-j}
\|\sig^{(j)}\|_{\infty}\leq \frac{c_3}{M+1},\qquad j=0,1,2,
\end{equation}
for some constant $c_3<\frac{c_2}{1+C(M)}\wedge c_1$, where $c_1$ and $C(M)$ are the constants in the statement of Theorem \ref{t1}, and $c_2$ the one of Proposition \ref{prop2}.

Let $z(x)=\{z_t,\;t\in [0,1]\}$ be the solution of Equation (\ref{4}) with control $x\in \cac^{\gam+\ep}$ and
diffusion coefficient $\sig_M$ (see (\ref{35}) and (\ref{36})).
Then, the map $x\mapsto z(x)$, defined in $\cac^{\gam+\ep}$ with values in $\cac^\kappa$ is Fr\'echet differentiable. Moreover, for all $h\in \cac^{\gam+\ep}$, the Fr\'echet derivative
of $z(x)$ is given by:
\beq
(Dz(x) \cdot h)_t = \int_0^1 \Phi_s(t) dh_s,
\label{23}
\eeq
where the kernels $\Phi_s(t)$ satisfy the following equation:
\beq
\Phi_s(t) =  \Psi_s(t) + G_M(x) \iou K(t,\xi) \si'(z_\xi)  \Phi_s(\xi) \, dx_\xi,
\label{58}
\eeq
with
\begin{equation}\label{24}
\Psi_s(t)   =  G_M(x) \, \sig(z_s) \, K(t,s)
+ 2 \ffi'_M(\|x\|^{2p}_{\gam,p}) \, \mu_s \, z_t,
\end{equation}
and
\begin{equation}\label{eq:def-mu}
\mu_s:=\int_0^s \int_s^1 \rho_{\zeta \eta} \, d\zeta d\eta,
\quad\mbox{where}\quad
\rho_{\zeta \eta} = 2p \frac{(x_\zeta-x_\eta)^{2p-1}}{|\zeta-\eta|^{2\gam p+2}}.
\end{equation}
\end{theorem}

\begin{proof} 
We will adapt the arguments used in the proof of Proposition 4 in \cite{NS}. That is, we will apply the Implicit Function Theorem to the functional $F$ defined in the statement of Proposition \ref{prop1}. For this, notice first that we have proved there that, for any $h\in \cac^{\gam+\ep}$ and $z\in \cac^\kappa$, $F(h,z)$ belongs to
$\cac^\kappa$ and $F$ is Fr\'echet differentiable with partial derivatives with respect to $h$ and $z$ given by (\ref{7}) and (\ref{8}), respectively. Moreover,
since $z$ is the solution of (\ref{4}), we have that $F(0,z)=0$.

\smallskip

We need to check now that $D_2F(0,z)$ defines a linear homeomorphism from $\cac^\kappa$ into itself for which, by the
Open Map Theorem, it suffices to prove that it is bijective (we already know that it is continuous). For this, we apply
Proposition \ref{prop2} to the case where $R_\xi=\sig'(z_\xi)$, so that
\beq
(D_2 F(0,z)\cdot g)_t = g_t -G_M(x) \int_0^1 K(t,\xi) \sig'(z_\xi) g_\xi\, dx_\xi
\label{26}
\eeq
defines a one-to-one mapping. Indeed, observe that condition (\ref{eq:hyp-sigma-j}) guarantees that (\ref{b-norm}) in Proposition
\ref{prop2} is satisfied.
On the other hand, if we fix $w\in \cac^\kappa$, applying again Proposition
\ref{prop2} we deduce that there exists $g\in \cac^\kappa$ such that $w=D_2 F(0,z)\cdot g$, which implies that
$D_2F(0,z)$ is onto and therefore a bijection.

\smallskip

Hence, by the Implicit Function Theorem, the map $x\mapsto z(x)$ is continuously Fr\'echet differentiable and
\beq
Dz(x)= - D_2F(0,z)^{-1} \circ D_1 F(0,z).
\label{25}
\eeq
Moreover, by (\ref{26}), for any $h\in \cac^{\gam+\ep}$, $Dz(x)\cdot h$ is the unique solution to the differential equation
\begin{equation*}
(Dz(x)\cdot h)_t  = w_t + G_M(x) \int_0^1 K(t,\xi) \sig'(z_\xi) (Dz(x)\cdot h)_\xi\, dx_\xi,
\end{equation*}
with $w_t=- (D_1 F(0,z)\cdot h)_t$.

\smallskip

Let us proceed to prove (\ref{23}). Consider Equation (\ref{58}) and integrate both sides with respect to some $h\in \cac^{\gam+\ep}$:
\begin{equation}
\label{27}
\int_0^1 \Phi_s(t) dh_s= \int_0^1 \Psi_s(t) dh_s  
+ G_M(x) \int_0^1 \left[ \int_0^1 K(t,\xi) \sig'(z_\xi) \Phi_s(\xi) dx_\xi \right] dh_s.
\end{equation}
At this point, we can use the same arguments as in the proof of Proposition 4 in \cite{NS}: apply our Fubini type Proposition \ref{prop:fubini-young} to  the last term in the right-hand side of (\ref{27}), which yields
\begin{equation}
\label{28}
\int_0^1 \Phi_s(t) dh_s= \int_0^1 \Psi_s(t) dh_s  
+ G_M(x) \int_0^1 K(t,\xi) \sig'(z_\xi) \left[ \int_0^1  \Phi_s(\xi) dh_s \right] dx_\xi.
\end{equation}
In order to conclude the proof, thanks to uniqueness part of Proposition \ref{prop2}, it is now sufficient to  show that $w=- D_1 F(0,z)\cdot h$ can be represented in the form
\beq
w_t=\int_0^1 \Psi_s(t) dh_s.
\label{29}
\eeq
For this, let us observe that, by (\ref{eq:compact-deriv1}), it holds:
\begin{equation*}
(D_1F(0,z)\cdot h)_t
=- G_M(x) \int_0^1 K(t,\xi) \sig(z_\xi) dh_\xi 
-(DG_M(x)\cdot h) \, z_t.
\end{equation*}
Hence, owing to Lemma \ref{l3} below, we obtain the representation (\ref{29}) with $\Phi_s(t)$ given by~(\ref{24}), which concludes the proof. 

\end{proof}

\smallskip

We close this section by giving an expression for $D G_M(x)$, which has already been used in the proof above.
\begin{lemma}\label{l3}
For all $h\in \cac^{\gam+\ep}$, it holds that
\beq
D G_M(x)\cdot h = 2 \ffi'_M(\|x\|^{2p}_{\gam,p}) \int_0^1 \mu_s \, dh_s,
\label{30}
\eeq
where the function $\mu$ is defined at equation (\ref{eq:def-mu}).
\end{lemma}

\begin{proof} 
As we have mentioned in Remark \ref{remark1}, the map $G_M:\cac^{\gam+\ep}\rightarrow \re$ is Fr\'echet
differentiable at any point $x\in \cac^{\gam+\ep}$, and its Fr\'echet derivative is given by:
$$DG_M(x)\cdot k = 2p \, \ffi_M'(\|x\|^p_{\gam,p}) \int_0^1\int_0^1
\frac{(x_\zeta-x_\eta)^{2p-1} (k_\zeta-k_\eta)}{|\zeta-\eta|^{2\gam p +2}} d\zeta d\eta, \quad k\in \cac^{\gam+\ep}.$$
According to the definition of $\rho_{\zeta \eta}$, this derivative can be written in the form: \beq DG_M(x)\cdot k =
\ffi_M'(\|x\|^p_{\gam,p}) \int_0^1\int_0^1 \rho_{\zeta \eta} \, (k_\zeta-k_\eta) d\zeta d\eta. \label{37} \eeq Then,
applying Fubini Theorem, one can argue as follows:
\begin{align*}
&\int_0^1\int_0^1 \rho_{\zeta \eta} \, (k_\zeta-k_\eta) d\zeta d\eta\\ 
& = \int_0^1\int_0^1 \rho_{\zeta \eta} \left(
\int_\eta^\zeta dk_r\right) \1_{\{\eta\leq \zeta\}}
d\zeta d\eta 
+ \int_0^1\int_0^1 \rho_{\zeta \eta} \left( \int_\eta^\zeta dk_r\right) \1_{\{\zeta\leq \eta\}} d\zeta d\eta \\
& = \int_0^1 \left[ \int_0^r d\zeta \int_r^1 d\eta \, \rho_{\zeta \eta} \right] dk_r 
+ \int_0^1 \left[ \int_0^r d\eta \int_r^1 d\zeta\, \rho_{\zeta \eta} \right] dk_r 
= 2 \int_0^1 \left[ \int_0^r d\zeta \int_r^1 d\eta \, \rho_{\zeta \eta} \right] dk_r.
\end{align*}
Plugging this expression in (\ref{37}) we obtain (\ref{30}) and we conclude the proof. 

\end{proof}


\section{Stochastic elliptic equations driven by a fractional Brownian motion}
\label{fbm}

Let us first describe the probabilistic setting in which we will apply the results obtained in the previous section.
For some fixed $H\in(0,1)$, we consider $(\Om,\tf,P)$ the canonical probability space associated with the fractional
Brownian motion with Hurst parameter $H$. That is,  $\Om=\cac_0([0,1])$ is the Banach space of continuous functions
vanishing at $0$ equipped with the supremum norm, $\tf$ is the Borel sigma-algebra and $P$ is the unique probability
measure on $\Om$ such that the canonical process $B=\{B_t, \; t\in [0,1]\}$ is a fractional Brownian motion with Hurst
parameter $H$. Remind that this means that $B$ is a centered Gaussian process with covariance
\[
R_H(t,s)=\frac 12 (s^{2H}+t^{2H}-|t-s|^{2H}).
\]
In particular, the paths of $B$ are $\gam$-H\"older continuous for all $\gam \in (0,H)$. Then, we consider Equation (\ref{4}) where
the driving trajectory is a path of $B$. Namely:
\[
\del z_{st} = \int_s^t du \left(\int_u^1 \sig_M(B,z_\xi) dB_\xi\right) - (t-s)\int_0^1 \xi \sig_M(B,z_\xi) dB_\xi, \quad 0\leq s<t\leq 1,
\]
which can be written in the reduced form 
\beq
z_t = G_M(B) \int_0^1 K(t,\xi) \sig (z_\xi) \, dB_\xi, \quad t\in [0,1].
\label{44} 
\eeq 
Assuming that $\sig\in \cac^2(\re)$ is bounded, has bounded derivatives and satisfies (\ref{sigma2}), Theorem~\ref{t1} implies that 
Equation (\ref{44}) has a unique solution $z=\{z_t,\; t\in [0,1]\}$ such that $z\in \cac^\kappa$ for any $\kappa\in (1-\gam,1)$, and almost surely in $\om\in\oom$.


\subsection{Malliavin differentiability of the solution}

This subsection is devoted to present the Malliavin calculus setting which we shall work in, so that we will be able to obtain that the 
solution of (\ref{44}) belongs to the domain of the Malliavin derivative. Notice that, in spite of the fact that we can solve  Equation (\ref{44}) driven by a fBm with arbitrary Hurst parameter, our Malliavin calculus section will be restricted to the range $H\in(1/2,1)$. This is due to the fact that stochastic analysis of fractional Brownian motion becomes cumbersome for $H<1/2$, and we have thus imposed this restriction for sake of conciseness.

\smallskip

Consider then a fixed parameter  $H>1/2$, and let us start by briefly describing the abstract 
Wiener space introduced for Malliavin calculus purposes (for a more general and complete description, we refer the reader to \cite[Section 3]{NS}). 

\smallskip

Let $\ce$ be the set of $\R$-valued step functions on $[0,1]$ and $\ch$ the completion of $\ce$ with respect to the semi-inner product
\[
\langle \1_{[0,t]}, \1_{[0,s]} \rangle_\ch := R_H(s,t), \qquad s,t \in [0,1].
\]
Then, one constructs an isometry $K^*_H: \ch \rightarrow  L^2([0,1])$  such that $K^*_H(\1_{[0,t]}) = \1_{[0,t]} K_H(t,\cdot)$, 
where the kernel $K_H$ is given by 
\[
K_H(t,s)= c_H s^{\frac 12 -H} \int_s^t (u-s)^{H-\frac 32} u^{H-\frac 12} \, du
\]
and verifies that $R_H(t,s)= \int_0^{s\land t} K_H(t,r) K_H(s,r)\, dr$, for some constant $c_H$. Moreover, let us observe that $K^*_H$ can be represented 
in the following form:
\[
[K^*_H \ffi]_t = \int_t^1 \ffi_r \partial_r K_H(r,t) \, dr.
\]
The fractional Cameron-Martin space can be introduced in the following way: let 
$\ck_H : L^2([0,1]) \rightarrow \ch_H := \ck_H(L^2([0,1]))$ be the operator defined by
\[
[\ck_H h](t) := \int_0^t K_H(t,s) h(s)\, ds, \qquad h\in L^2([0,1]).
\]
Then, $\ch_H$ is the Reproducing Kernel Hilbert space associated to the fractional Brownian motion $B$. Observe that, in the case of the 
classical Brownian motion, one has that $K_H(t,s)=\1_{[0,t]}(s)$, $K^*_H$ is the identity operator in $L^2([0,1])$ and $\ch_H$ is the usual Cameron-Martin 
space. 

\smallskip

In order to deduce that $(\Om,\ch,P)$ defines an abstract Wiener space, we remark that $\ch$ is continuously and densely embedded in $\Om$. In fact, 
one proves that the operator $\crr_H :\ch \rightarrow \ch_H$ given by
\[
\crr_H \psi := \int_0^\cdot K_H(\cdot,s) [K^*_H \psi](s)\, ds
\]
defines a dense and continuous embedding from $\ch$ into $\Om$; this is due to the fact that $\crr_H \psi$ is $H$-H\"older continuous (for details, see 
\cite[p. 9]{NS}).  

\smallskip

At this point, we can introduce the Malliavin derivative operator on the Wiener space $(\Om,\ch,P)$. Namely, we first let $\cs$ be the 
family of smooth functionals $F$ of the form
\[
F=f(B(h_1),\dots,B(h_n)),
\]
where $h_1,\dots,h_n\in \ch$, $n\geq 1$, and $f$ is a smooth function having polynomial growth together with all its  partial derivatives. 
Then, the Malliavin derivative of such a functional $F$ is the $\ch$-valued random variable defined by
\[
\cd F= \sum_{i=1}^n \frac{\partial f}{\partial x_i} (B(h_1),\dots,B(h_n)) h_i.
\]
For all $p>1$, it is known that the operator $\cd$ is closable from $L^p(\Om)$ into $L^p(\Om; \ch)$ (see e.g. \cite[Section 1]{nualart}).
We will still denote by $\cd$ the closure of this operator, whose domain is usually denoted by $\D^{1,p}$ and is defined 
as the completion of $\cs$ with respect to the norm
\[
\|F\|_{1,p}:= \left( E(|F|^p) + E( \|\cd F\|_\ch^p ) \right)^{\frac 1p}.
\]
The local property of the operator $\cd$ allows to define the localized version of $\D^{1,p}$, as follows. 
By definition, $F\in \D^{1,p}_{loc}$ if there is a sequence $\{(\Om_n,F_n),\; n\geq\}$ in $\tf\times \D^{1,p}$ such that
$\Om_n$ increases to $\Om$ with probability one and $F=F_n$ on $\Om_n$. In this case, one sets $\cd F:= \cd F_n$ on $\Om_n$.

We will first prove now that the solution of (\ref{44}) at any $t\in [0,1]$ belongs to $\D^{1,p}_{loc}$. For this, we need to introduce
the notion of differentiability of a random variable $F$ in the directions of $\ch$, and we shall apply a classical result 
of Kusuoka (see \cite{kusuoka} or \cite[Proposition 4.1.3]{nualart}). Indeed, a random variable $F$ is $\ch$-differentiable if,
by definition, for almost all $\om\in \Om$ and for any $h\in \ch$, the map $\nu \mapsto F(\om + \nu \crr_H h)$ is differentiable. 
Then, the above-mentioned result of Kusuoka states that any $\ch$-differentiable random variable $F$ belongs to the space 
$\D^{1,p}_{loc}$, for any $p>1$. We have the following result:

\begin{proposition}\label{prop3}
Let $\gam,\kappa\in (0,1)$ be such that $\gam+\kappa>1$. Let $\ep>0$ and a sufficiently large $p\geq 1$ so that $\ep>\frac{1}{2p}$ and $\gam+\ep<H$ (this latter condition guarantees that $B\in \cac^{\gam+\ep}$). Assume that $\sig$ satisfies the hypotheses of Theorem \ref{t2}.
 
Let $z=\{z_t,\; t\in [0,1]\}\in \cac^\kappa$ be the unique solution of equation (\ref{44}). Then, for any $t\in [0,1]$, $z_t\in \D^{1,2}_{loc}$ and we have:
\beq
\langle \cd z_t,h\rangle_\ch = \left[ D z(B) (\crr_H h)\right]_t, \qquad h\in \ch.
\label{56}
\eeq
\end{proposition}

\begin{proof}
Recall that the process $B$ is $\gam$-H\"older continuous for any $\gam\in (0,H)$. Hence, in the statement of Theorem \ref{t2}, we will 
be able to find $\ep$ (choosing $p$ therein sufficiently large) such that $\gam+\ep<H$ and $\|B\|_{\ga,p}$ is finite almost surely. 

\smallskip

On the other hand, note that for all $h\in \ch$, we have:
\[
|(\crr_H h)(t)-(\crr_H h)(s)| = \left( E (|B_t - B_s|^2) \right)^{\frac 12} \|h\|_\ch \leq |t-s|^H \, \|h\|_\ch.
\]
Consequently, by Theorem \ref{t2} and Lemma \ref{l4} below, we can infer that $z_t$ is $\ch$-dif\-fe\-ren\-tia\-ble. Therefore, Kusuoka's result
implies that $z_t\in \D^{1,2}_{loc}$ and we have
\beq
\langle \cd z_t,h\rangle_\ch = \frac{d}{d\nu} z_t (\om + \nu \crr_H h)_{\big|_{\nu=0}}\quad a.s.,
\label{57}
\eeq
which, together with Lemma \ref{l4}, allows us to conclude that 
\[
\langle \cd z_t,h\rangle_\ch = D z_t (B) (\crr_H h) = \left[ D z(B) (\crr_H h)\right] (t).
\]
\end{proof}

\begin{lemma}\label{l4}
Let $\gam<H$ and $\ep>0$ such that $\gam + \ep<H$, as in the statement of Theorem~\ref{t2}. 
Let $z$ be the solution of (\ref{44}) and $t\in [0,1]$. Then $x\mapsto z_t(x)$ is Fr\'echet differentiable 
from $\cac^{\gam+\ep}$ into $\R$. Furthermore, for $x\in \cac^{\gam+\ep}$, it holds:
\[
D z_t(x)(k)= \left[ D z (x)(k)\right]_t, \qquad k\in \cac^{\gam+\ep}.
\]
\end{lemma}

\begin{proof}
It is very similar to that of \cite[Lemma 4.2]{LT}. Indeed, the following estimates are readily checked:
\begin{align*}
|z(x+k)_t - z(x)_t - [D z(x) k]_t| & \leq \|z(x+k) - z(x) - [D z(x) k]\|_\infty \\
& \leq \|z(x+k) - z(x) - [D z(x) k]\|_{\gam+\ep}.
\end{align*}
In addition, Theorem \ref{t2} ensures that the latter term is of order $o(\|k\|_{\gam+\ep})$, from which our claim is easily deduced.
 
\end{proof}

At this point, let us go a step further and prove that the solution $z_t$ of Equation (\ref{44}), indeed, belongs to $\D^{1,2}$.  

\begin{proposition}\label{prop4}
Let $\gam,\kappa\in (0,1)$ be such that $\gam+\kappa>1$. Let $\ep>0$ and a sufficiently large $p\geq 1$ so that $\ep>\frac{2}{p}$ and $\gam+\ep<H$. Assume that $\sig$ satisfies the hypotheses of Theorem \ref{t2}.

Let  $z=\{z_t,\; t\in [0,1]\}$ be the unique solution of equation (\ref{44}).
Then, for any $t\in [0,1]$, $z_t$ belongs to $\D^{1,2}$.
\end{proposition}

\begin{proof}
By (\ref{56}), formula (\ref{23}) and the definition and properties of $\crr_H$, we have the following equalities:
for any $h\in \ch$,
\begin{align*}
\langle \cd z_t,h\rangle_\ch & = [Dz(B) (\crr_H h)]_t = \int_0^1 \Phi_s(t)\, d(\crr_H h)_s \\
& = \int_0^1 \Phi_s(t) \left( \int_0^s \frac{\partial K_H}{\partial s} (s,r) (K^*_H h)(r)\, dr\right) ds \\
& = \int_0^1 (K^*_H \Phi_\cdot (t))(s) (K^*_H h)(s)\, ds = \langle \Phi_\cdot (t), h\rangle_\ch. 
\end{align*}
This implies that, as elements of $\ch$, $\cd z_t = \Phi_\cdot (t)$. 

\smallskip

On the other hand, let us observe that $L^{\frac 1H}([0,1])\subset \ch$ continuously (see e.g. \cite[Lemma 5.1.1]{nualart}), and clearly any H\"older space $\cac^\ka$ is continuously embedded in $L^{\frac 1H}([0,1])$. Therefore, if we aim to prove that
$E(\|\cd z_t \|^2_\ch) < +\infty$, it suffices to verify that $E(\|\Phi_\cdot(t)\|_\ka)<\infty$, for any $\ka\in (0,1)$. 

\smallskip

Taking into account that $\Phi_s(t)$ satisfies 
the linear equation (\ref{58}), we are in position to apply Proposition \ref{prop2} so that we end up with 
\beq
\|\Phi_\cdot(t)\|_\ka \leq C(M) \|\Psi_\cdot(t)\|_\ka,
\label{eq:bound-phi}
\eeq
where we remind that $\Psi_s(t)$ has been defined in (\ref{24}). By the boundedness of $G_M$ and $\varphi'_M$, the fact that $K(t,\cdot)$ is Lipschitz with Lipschitz constant bounded by $1-t$, Lemma~\ref{l1} and estimate (\ref{eq:bound-z}), we can infer that
\beq
\|\Psi_\cdot(t)\|_\ka \leq C (1+\|\mu\|_\ka),
\label{eq:bound-psi}
\eeq
for some constant $C$ depending on $M$ and $\sig$. Hence, it remains to study the $\ka$-H\"older regularity of $\mu$ (recall that this process is defined by (\ref{eq:def-mu})). Namely, for any $0\leq s_1<s_2\leq 1$, one easily verifies that
\[
\mu_{s_2}-\mu_{s_1} =\int_{s_1}^{s_2} \int_{s_2}^1 \rho_{\zeta \eta} \, d\zeta d\eta
- \int_0^{s_1} \int_{s_1}^{s_2}  \rho_{\zeta \eta} \, d\zeta d\eta.
\]
At this point, let us observe that, in the statement, the condition  relating $p$ and $\ep$ is slightly stronger than the one considered in Proposition \ref{prop3}. In fact, the former allows us to infer that $\rho_{\zeta \eta}\leq C \|B\|_{\gam+\ep}^{2p-1}$, almost surely, which guarantees that 
$\mu\in \cac^\ka$ and 
\beq
\|\mu\|_\ka \leq C \|B\|_{\gam+\ep}^{2p-1}.
\label{eq:bound-mu}
\eeq
Plugging this bound in (\ref{eq:bound-psi}) and using (\ref{eq:bound-phi}), we end up with 
\[
E(\|\cd z_t \|^2_\ch) \leq E(\|\Phi_\cdot(t)\|_\ka^{2}) \leq C E(\|B\|_{\gam+\ep}^{4p-2}),
\]
and the latter is a finite quantity since $\gam+\ep<H$ and $\|B\|_{\gam+\ep}$ has moments of any order by Fernique's lemma \cite[Theorem 1.2.3]{Fe}. This concludes the proof.

\end{proof}


\subsection{Stratonovich interpretation of the fractional elliptic equation} \label{sec:strato}

Up to now, we have succeeded in solving equation (\ref{44}) by interpreting any integral with respect to $B$ in the Young (pathwise) sense. In this particular situation, it is a well known fact \cite{rv} that our approach is equivalent to Russo-Vallois kind of techniques. Namely, if for a process $V$ the integral $\int_0^T V_s \, dB_s$ can be defined in the Young sense, then one also has almost surely
$$
\int_0^T V_s \, dB_s=
\lim_{\ep\to 0} \frac{1}{2\ep} \int_0^T V_s \, \lp  B_{s+\ep}-B_{s-\ep} \rp \, ds.
$$
The latter limit is usually called Stratonovich integral with respect to $B$ (see \cite[Definition 5.2.2]{nualart}), and is denoted by $\int_0^T V_s \circ dB_s$.

\smallskip

Our point of view in this section is slightly different: we wish to show that the integrals with respect to $B$ in equation~(\ref{44}) can also be interpreted as the sum of a Skorohod integral plus a trace term. As we shall see below (see Proposition \ref{prop:dom-strato}), this gives another definition of Russo-Vallois symmetric integral in the particular case of smooth integrands in the Malliavin calculus sense. In particular we shall see that, at least a posteriori, Malliavin calculus might have been applied in order to solve our original elliptic equation, though a direct application of these techniques lead to non closed estimations.

\smallskip 

Let us thus introduce the space $|\ch|$, which is composed of measurable functions $\vp:[0,1]\rightarrow \R$ such that
\[
\|\vp\|^2_{|\ch|}:= \al_H \int_0^1 \int_0^1 |\vp_r| |\vp_u| |r-u|^{2H-2} dr du <+\infty,
\]
where $\al_H=H(2H-1)$, and we denote by $\langle \cdot,\cdot\rangle_{|\ch|}$ the associated inner product. We define Stratonovich integrals thanks to the following result, borrowed from \cite[Proposition 3]{alos-nualart}:
\begin{proposition}\label{prop:dom-strato}
Let $\{u_t,\; t\in [0,1]\}$ be a stochastic process in $\D^{1,2}(|\ch|)$ such that
\beq
\int_0^1 \int_0^1 |\cd_s u_t| |t-s|^{2H-2} ds dt <+\infty \quad a.s.
\label{eq:an}
\eeq
Then, the Stratonovich integral $\int_0^1 u_t \circ dB_t$ exists and can be written as
\begin{equation}\label{eq:sko-trace}
\int_0^1 u_t \circ dB_t = \del(u) + \int_0^1 \int_0^1 \cd_s u_t |t-s|^{2H-2} ds dt,
\end{equation}
where $\del(u)$ stands for the Skorohod integral of $u$.
\end{proposition}
We are now in a position to apply this result to our elliptic equation:
\begin{proposition}\label{prop:strato}
Let $z=\{z_t,\; t\in [0,1]\}$ be the solution to equation (\ref{44}).
Under the same hypothesis as in Proposition \ref{prop4}, the process $z$ belongs to $\D^{1,2}(|\ch|)$ and also satisfies the equation
\beq
z_t = G_M(B) \int_0^1 K(t,\xi) \sig (z_\xi) \circ dB_\xi, \quad t\in [0,1],
\label{eq:strato}
\eeq
where the Stratonovich stochastic integral with respect to $B$ is interpreted as in (\ref{eq:sko-trace}).
\end{proposition}

\begin{proof}
Note first that the norm of $z$ in $\D^{1,2}(|\ch|)$ is given by
\[
\|z\|^2_{\D^{1,2}(|\ch|)} = E(\|z\|^2_{|\ch|}) + E(\|\cd z\|^2_{|\ch|\otimes |\ch|}).
\]
By Theorem \ref{t1} (see (\ref{eq:bound-z}) therein), we have 
\begin{align}
E(\|z\|^2_{|\ch|}) & = \int_0^1 \int_0^1 E(|z_r z_u|) |r-u|^{2H-2} dr du \nonumber \\
& \leq C E (\|z\|_{\infty}^2) \int_0^1 \int_0^1 |r-u|^{2H-2} dr du \leq C.
\label{eq:s1}
\end{align}
On the other hand, owing to (\ref{eq:bound-phi})-(\ref{eq:bound-mu}) we can infer that, for any $r,u\in [0,1]$:
\[
E(|\cd_u z_r|)\leq C,
\]
for some positive constant $C$. Thus
\begin{align}
&E(\|\cd z\|^2_{|\ch|\otimes |\ch|})\\ 
& =  \int_0^1 \int_0^1 dr_1 dr_2 \, |r_1-r_2|^{2H-2}    
\int_0^1 \int_0^1 du_1 du_2 \, E\left(|\cd_{r_1} z_{u_1}|\, |\cd_{r_2} z_{u_2}| \right) |u_1-u_2|^{2H-2}  \nonumber \\
& \leq C \int_0^1 \int_0^1 dr_1 dr_2 \, |r_1-r_2|^{2H-2} \int_0^1 \int_0^1 du_1 du_2 \, |u_1-u_2|^{2H-2} < +\infty.
\label{eq:s2}
\end{align}
Putting together (\ref{eq:s1}) and (\ref{eq:s2}), we have seen that $z\in \D^{1,2}(|\ch|)$, and one also deduces that (\ref{eq:an}) holds. By Proposition \ref{prop:dom-strato}, this implies that $z$ belongs to the domain of the Stratonovich integral. Therefore, thanks to the regularity properties of $\sig$ and the fact that 
$K(t,\cdot)$ is a deterministic function, we obtain that the Stratonovich integral $\int_0^1 K(t,\xi) \sig (z_\xi) \circ dB_\xi$ is well-defined. By 
\cite[Section 2.2, Proposition 3]{rv-LNM}, this Stratonovich integral coincides with the pathwise Young integral on the right-hand side of (\ref{44}), for which 
we can conclude that $z$ solves (\ref{eq:strato}). 

\end{proof}


\subsection{A modified elliptic equation}
\label{density}

One of the major obstacles on our way to get the absolute continuity of $\cl(z_t)$ is the following: associated to equation (\ref{44}) is the process $\mu$ defined by (\ref{eq:def-mu}), appearing in the expression for $\cd z_t$. This process happens to have some fluctuations around $s=0$ which are too high to guarantee the strict positivity of $\cd z_t$ at least in a small interval. This is why we consider in this section a slight modification of our elliptic equation (\ref{44}) and we will prove that its solution, at any instant $t$, has a law which is absolutely continuous with respect to the Lebesgue measure. Specifically, the cutoff term $G_M(B)$ in equation (\ref{44}) will be replaced by a new $\tilde G_M(B)$, whose motivation relies on a variation of Garsia's lemma given below:

\begin{proposition}\label{prop:4.4}
Let $f$ be a continuous function defined on $\ou$. Set, for $p\ge 1$,
\begin{equation}
U_{\gam,p}(f):=\lp\int_0^1 dv \int_{v}^{4v\wedge 1} \frac{|{\del f}_{uv}|^{2p}}{|v-u|^{2\gam p+2}}
\ du \rp^{1/2p},
\label{eq:def-u}
\end{equation}
and assume $U_{\gam,p}(f)<\infty$. Then $f \in C^{\gam}(\ou)$; more precisely,
\begin{equation}
\|f\|_{\gam} \le c \, U_{\gam,p}(f),
\end{equation}
for a universal constant $c>0$.
\end{proposition}

\begin{proof}

Let $0\le s<t\le 1$. We wish to show that 
\begin{equation}\label{eq:50}
|\del f_{st}|\le c \, U_{\gam,p}(f)\,  |t-s|^{\gam}.
\end{equation} 
To this end, let us construct a sequence of
points $(s_k)_{k\ge 0}$, $s_k\in\ou$, converging to $t$ in the following way: set  $s_0=s$, suppose by induction that
$s_0,\ldots, s_k\leq t$ have been constructed, and let 
\begin{equation}\label{eq:def-V-k}
V_k:=[a_k,b_k],
\quad\mbox{with}\quad
a_k=2s_{k}\wedge \lp  \frac{s_k+t}{2}\rp, \
b_k=3s_{k}\wedge t.
\end{equation}
Notice that the main differences between our proof an the original one by Garsia (or better said the one given by Stroock in \cite{St}) stems from this definition of $a_k,b_k$. Indeed, in the classical proof, $a_k=\frac{s_k+t}{2}$ and $b_k=t$.
Define then
\begin{equation}
A_k:=\left\{v\in V_k\ |\ I(v) > \frac{6 \, U_{\gam,p}^{2p}(f)}{|v-s_k|} \right\} \label{eq:Ak}
\end{equation}
and
\begin{equation}
B_k:=\left\{v\in V_k\ |\ \frac{|{\del f}_{s_k v}|^{2p}}{|v-s_k|^{2\gam p+2}}>\frac{6 \, I(s_k)}{|v-s_k|}
 \right\} \label{eq:Bk}
\end{equation}
where we have set
\begin{equation*} I(v):=\int_{v}^{t\wedge 4v} \frac{|{\del f}_{uv}|^{2p}}{|v-u|^{2\gam p+2}}
\ du.
\end{equation*}
Let us prove now that $V_k\setminus (A_k\cup B_k)$ is not empty: observe that, for $t\in\ou$,
$$
\int_{v}^{4v\wedge 1} \frac{|{\del f}_{uv}|^{2p}}{|v-u|^{2\gam p+2}}
\ du 
\ge \int_{v}^{4v\wedge t} \frac{|{\del f}_{uv}|^{2p}}{|v-u|^{2\gam p+2}} \ du 
=I(v),
$$
and thus
\begin{equation*}
U_{\gam,p}^{2p}(f) \ge \int_{A_k} I(v) \, dv >   \frac{6 \, U_{\gam,p}^{2p}(f)}{|b_k-s_k|} \mu(A_k). 
\end{equation*}
Moreover,
\begin{multline*}
I(s_k)
=\int_{s_k}^{4s_k\wedge t} \frac{|{\del f}_{s_ku}|^{2p}}{|u-s_k|^{2\gam p+2}} \, du
\ge \int_{B_k} \frac{|{\del f}_{s_ku}|^{2p}}{|u-s_k|^{2\gam p+2}} \, du  \\
> \int_{B_k} \frac{6 \, I(s_k)}{|u-s_k|} \, du
\ge \frac{6\, I(s_k)}{|b_k-s_k|} \, \mu(B_k).
\end{multline*}
All together one has obtained $\mu(A_k),\mu(B_k)<\frac{|b_k-s_k|}{6}$, so that $\mu(A_k)+\mu(B_k)<\frac{|b_k-s_k|}{3}$.

\smallskip

Next we show that $|b_k-s_k|=2\mu(V_k)=2|b_k-a_k|$. This study can be separated in two cases: 

\smallskip

\noindent
\textit{(i)} If $s_k\le t/3$, then $a_k=2s_k$ and $b_k=3s_k$. Thus $b_k-a_k=s_k$ and $b_k-s_k=2s_k$. This obviously yields $|b_k-s_k|=2\mu(V_k)$.

\smallskip

\noindent
\textit{(ii)} If $s_k> t/3$, then $a_k=\frac{s_k+t}{2}$ and $b_k=t$. Thus $b_k-a_k=\frac{t-s_k}{2}$ and $b_k-s_k=t-s_k$. Here again, we get $|b_k-s_k|=2\mu(V_k)$.

\smallskip

We have thus proved that $\mu(A_k)+\mu(B_k)<\frac{2\mu(V_k)}{3}$, which means that $V_k\setminus (A_k\cup B_k)$ is not empty. Let us thus choose $s_{k+1}$ arbitrarily in this set. Note that, by construction,
$s_k\to t$ while staying inside $[s,t]$.

\smallskip

Now, for an arbitrary $n\ge 1$, decompose ${\del f}_{st}$ into
\begin{equation} \label{eq:dec-del}  
{\del f}_{st}={\del f}_{s_{n+1} t}+\sum_{k=0}^n 
{\del f}_{s_k s_{k+1}} .
\end{equation}
Applying (\ref{eq:Bk})$_k$ and (\ref{eq:Ak})$_{k-1}$, one gets
\begin{equation*}
\frac{ |{\del f}_{s_k s_{k+1}}|^{2p}}{|s_{k+1}-s_k|^{2\gam p+2}} 
\le  \frac{c \, I(s_k)}{|s_{k+1}-s_{k}|}
\le 
\frac{c\, U_{\gam,p}^{2p}(f)  }{|s_{k+1}-s_{k}| |s_{k}-s_{k-1}|},
\end{equation*}
and hence
\begin{equation}\label{eq:garsia-1}
|{\del f}_{s_k s_{k+1}}|^{2p} 
\le c \,  Q_k \, U_{\gam,p}^{2p}(f) \, |s_{k+1}-s_k|^{2\gam p},
\quad\mbox{where}\quad
Q_k:=\frac{|s_{k+1}-s_{k}|}{|s_{k}-s_{k-1}|}.
\end{equation}
Notice that in our definition (\ref{eq:def-V-k}), we have $a_k=2s_k$ instead of $(s_k+t)/2$ iff $s_k<t/3$. 
Therefore, we can distinguish three cases in order to bound the quantity $Q_k$ above:

\smallskip

\noindent
\textit{(i)} If $s_{k-1}>t/3$, then $s_{k+1}-s_k\le t-s_k$ and $s_k-s_{k-1}\ge (t-s_k)/4$. Thus $Q_k\le 4$. 

\smallskip

\noindent
\textit{(ii)} If $s_k\le t/3$, then $s_{k+1}-s_k\le 3s_k-s_k=2s_k$ and $s_k-s_{k-1}\ge s_k-s_k/2=s_k/2$. Thus $Q_k\le 4$ again. 

\smallskip

\noindent
\textit{(iii)} If $s_{k-1}\le t/3$ and $s_{k}>t/3$, then $s_{k+1}-s_k\le t-s_k\le 3s_k-s_k=2s_k$ and $s_k-s_{k-1}\ge s_k/2$. Thus $Q_k\le 4$.

\smallskip

\noindent
Putting those estimates together, we end up with $Q_k\le 4$ in all cases, and plugging this inequality into (\ref{eq:garsia-1}), we obtain
\begin{equation*}
|{\del f}_{s_k s_{k+1}}|^{2p} 
\le c \,   U_{\gam,p}^{2p}(f) \, |s_{k+1}-s_k|^{2\gam p}.
\end{equation*}
Now (\ref{eq:dec-del}) reads
\begin{equation}\label{eq:dcp-del-2}
\lln {\del f}_{st}\rrn \le \lln {\del f}_{s_{n+1} t}\rrn +\sum_{k=0}^n 
\lln {\del f}_{s_k s_{k+1}}\rrn
\le  \lln {\del f}_{s_{n+1} t}\rrn + U_{\gam,p}(f)\sum_{k=0}^n |s_{k+1}-s_k|^{\gam}.
\end{equation}

\smallskip

It remains to bound $\sum_{k=0}^n |s_{k+1}-s_k|^{\gam}$ for an arbitrary $n$. This is achieved by separating cases again:

\smallskip

\noindent
\textit{(i)} If $s>t/3$, then it is easily shown that $a_k=\frac{s_k+t}{2}$ and $b_k=t$, for all $k$, for which we have $s_{k+1}\in [(s_k+t)/2,t]$. This implies that 
$t-s_{k+1}\leq (t-s_k)/2$ and hence $t-s_k\leq 2^{-k} (t-s)$, for any $k\geq 0$. Therefore
\[
 s_{k+1}-s_k\leq t-\frac{t+s_{k-1}}{2}=\frac{t-s_{k-1}}{2}\leq \frac{t-s}{2^k}.
\]
Plugging this into (\ref{eq:dcp-del-2}):
\[
 |\del f_{s t}|\leq |{\del f}_{s_{n+1} t}| + U_{\gam,p}(f)\sum_{k=0}^n \frac{1}{2^{\gam k}} |t-s|^\gam.
\]
Let $n\rightarrow \infty$ and use the continuity of $f$ and the fact that $s_{n+1}\rightarrow t$. Then
\[
 |\del f_{s t}|\leq C U_{\gam,p}(f) |t-s|^\gam,
\]
where $C$ denotes a positive constant which may depend on $\gam$. This concludes the proof in the case $s>t/3$.

\smallskip

\noindent
\textit{(ii)} If $s\le t/3$, then by definition of $s_{k+1}$ we will have, for small enough $k$, that $s_{k+1}=\beta_{k+1} s_k$ for some
$\beta_{k+1}\in [2,3]$. Thus
\[
s_k=\left( \prod_{j=1}^k \beta_j\right) s.
\]
Set $M:=\inf\{k\in \N;\, \prod_{j=1}^k \beta_j \geq t/(3s)\}$, so that we wish to evaluate $\sum_{k=0}^{M-1}|s_{k+1}-s_k|^\gam$:
\[
\sum_{k=0}^{M-1}|s_{k+1}-s_k|^\gam  = \sum_{k=0}^{M-1} \left( \prod_{j=1}^k \beta_j s\right)^\gam (\beta_{k+1}-1)^\gam 
 \leq 2^\gam s^\gam \sum_{k=0}^{M-1} \left( \prod_{j=1}^k \beta_j \right)^\gam.
\]
Notice that $b_M:=\prod_{j=0}^{M-1} \beta_j \leq t/(3s)$, by definition of $M$, and 
\[
\prod_{j=0}^{M-l} \beta_j \leq \frac{b_M}{2^{l-1}},\qquad \text{for} \quad l=1,\dots,M,
\]
since $\beta_j\geq 2$. Hence
\[
\sum_{k=0}^{M-1} \left(\prod_{j=0}^k \beta_j\right)^\gam \leq \sum_{l=1}^M \frac{b_M^\gam}{2^{(l-1)\gam}} < C b_M^\gam \leq C (t/s)^\gam
\]
and 
\[
\sum_{k=0}^{M-1} |s_{k+1}-s_k|^\gam \leq 2^\gam s^\gam C (t/s)^\gam\leq C t^\gam.
\]
Observe that $t-s>2t/3$ whenever $s<t/3$. Therefore $t<\frac 32 (t-s)$ and 
\[
 \sum_{k=0}^{M-1} |s_{k+1}-s_k|^\gam \leq C (t-s)^\gam.
\]

\smallskip

Let us go back now to (\ref{eq:dcp-del-2}) and write
\begin{eqnarray}\label{eq:dcp-del-3}
\lln {\del f}_{st}\rrn &\le& \lln {\del f}_{s_{n+1} t}\rrn 
+ U_{\ga,p}(f) \lp \sum_{k=0}^{M-1} |s_{k+1}-s_{k}|^{\ga} + \sum_{k=M}^{n} |s_{k+1}-s_{k}|^{\ga} \rp  
\notag\\
&:=& \lln {\del f}_{s_{n+1} t}\rrn + U_{\ga,p}(f) \lp A_M+B_M \rp.
\end{eqnarray}
We have just seen that $A_M\le C (t-s)^\gam$, and one can also prove that $B_M\le C (t-s)^\gam$ uniformly in $n$ by means of the same kind of argument as for step (i). This ends the proof by taking limits in (\ref{eq:dcp-del-3}).

\end{proof}

We will now take advantage of the previous proposition in order to build a slight modification of our elliptic equation which is amenable to density results. Namely, as before, 
let $M>0$ be any real number and $\ffi_M\in \cac^\infty_b((0,\infty))$ such that $\ffi_M(r)=0$
for all $r>M+1$, and $\ffi_M(r)=1$ for $r<M$. For any $x:[0,1]\rightarrow \re$ for which $U_{\gam,p}(x)<\infty$, for
some $\gam\in (0,1)$ and $p\geq 1$, set 
\[ 
\tilde{G}_M(x):=\ffi_M(U_{\gam,p}(x)^{2p}),
\]
and, for such $x$ and any $z\in \re$, we define 
\beq
\tilde{\sig}_M(x,z):= \tilde{G}_M(x) \sig(z).
\label{eq:sigma-u}
\eeq
We shall thus consider another kind of modified elliptic integral equation driven by the fractional Brownian motion $B$:
\beq
\del z_{st} = \int_s^t du \left(\int_u^1 \tilde{\sig}_M(B,z_\xi) dB_\xi\right) - (t-s)\int_0^1 \xi \tilde{\sig}_M(B,z_\xi) dB_\xi, \quad 0\le s\le t\le 1.
\label{eq:ell-u}
\eeq
This equation can be equivalently formulated in its compact form:
\beq
z_t = \tilde G_M(B) \int_0^t K(t,\xi) \sig(z_\xi) dB_\xi, \qquad t\in [0,1]. 
\label{eq:ell-u-tilde}
\eeq
We will prove that the probability law of the solution to (\ref{eq:ell-u-tilde}) taken at $t\in(0,1)$ is absolutely continuous with respect to the Lebesgue measure.

\smallskip

First of all, let us point out that the results of Sections \ref{existence} and \ref{differentiability} remain valid for the solution of Equation (\ref{eq:ell-u-tilde}). Moreover, using exactly the same arguments as in the proof of Proposition \ref{prop4}, one obtains that, for all $t\in [0,1]$, the solution $z_t$ belongs to the domain of the Malliavin derivative. Altogether, we can state the following result:

\begin{theorem}\label{thm:ell-tilde}
Let $\gam,\kappa\in (0,1)$ be such that $\gam+\kappa>1$. Let $\ep>0$ and a sufficiently large $p\geq 1$ so that $\ep>\frac{2}{p}$ and $\gam+\ep<H$. Assume that $\sig$ satisfies the hypotheses of Theorem~\ref{t2}.

Then, there exists a unique solution $z=\{z_t,\; t\in [0,1]\}$ of (\ref{eq:ell-u-tilde}), which is an element of $\cac^\kappa$. For any $t\in [0,1]$, 
$z_t$ belongs to $\D^{1,2}$ and the Malliavin derivative $\cd z_t$ satisfies the following linear integral equation:

\beq
\cd_s z_t =  \Psi_s(t) + \tilde{G}_M(B) \iou K(t,\xi) \si'(z_\xi)  \cd_s z_\xi \, dB_\xi,\quad s\in [0,1],
\label{eq:malder}
\eeq
with
\begin{equation}\label{eq:Psi-t}
\Psi_s(t)   =  \tilde{G}_M(B) \, \sig(z_s) \, K(t,s)
+ 2 \ffi'_M(U_{\gam,p}(B)^{2p}) \, \tilde{\mu}_s \, z_t,
\end{equation}
and
\begin{equation}
\tilde{\mu}_s:=\int_{\frac s4}^s \int_s^{4\eta\wedge 1} \rho_{\zeta \eta} \, d\zeta d\eta,
\quad\mbox{where}\quad
\rho_{\zeta \eta} = (2p-1) \frac{(B_\zeta-B_\eta)^{2p-1}}{|\zeta-\eta|^{2\gam p+2}}.
\label{eq:tilde-mu-rho}
\end{equation}

\end{theorem}

\begin{remark}\label{rmk:diff-sigma-u}
The term $\tilde \mu_s$ in (\ref{eq:Psi-t}) comes from the fact that, as one can easily verify, the Fr\'echet derivative of $\tilde{G}_M$ at $x\in \cac^{\gam+\ep}$ is given by
\[
D\tilde{G}_M(x)\cdot h=2 \ffi_M'(U_{\gam,p}(x)^{2p}) \int_0^1 \tilde{\mu}_s dh_s,\qquad h\in \cac^{\gam+\ep}. 
\]
\end{remark}

\smallskip

We can now give the technical justification for our change in the elliptic equation we consider: the lemma below (whose proof can be immediately deduced from (\ref{eq:tilde-mu-rho})) shows that $\tilde{\mu}$ can be made of order $s^q$ for an arbitrary large $q$ and $s$ in a neighborhood of 0. This simple fact will enable us to upper bound $|\cd_s z_t|$ for $s\to 0$ in a satisfying way. The following result will thus be important in the sequel:
\begin{lemma}\label{lemma:mu}
Assume that the hypothesis of Theorem \ref{thm:ell-tilde} are satisfied. 
Then, for all $s\in (0,\frac 14)$ and $p\ge 1$:
\beq
\tilde{\mu}_s \leq \|B\|_{\gam +\ep}^{2p-1} \, s^\beta \qquad a.s.,
\label{eq:bound-mu-del}
\eeq
where $\beta=(2p-1)\ep -\gam$. 
\end{lemma}

Fix $t\in (0,1)$, and consider $z_t$ solution to (\ref{eq:ell-u-tilde}). Observe that the random variable $z_t$ cannot have a density $p_t(y)$ at $y=0$, since $P(z_t=0)>0$ due to our cutoff procedure. Hence we will prove the existence of density for the law of the random variable $z_t$ on the subset of $\Omega$ defined by
$\Om_a:=\{|z_t|\geq a\}$, for all $a>0$. The fact that we are restricting our analysis to $\Om_a$ implies the following simple but useful properties:

\begin{lemma}\label{lemma:omega-a}
On $\Om_a$, we have
\[
\|B\|_\gam\leq C_1 \qquad \text{and}\qquad \tilde{G}_M(B)\geq C_2,\quad a.s.
\]
where $C_1, C_2$ denote some positive constants depending on $a$ and $M$.
\end{lemma}

\begin{proof}
Note that on $\Om_a$ we must clearly have that 
\[
\tilde{G}_M(B)=\ffi_M(U_{\gam,p}(B)^{2p})>0 \qquad \mbox{a.s. }
\]
Thus, by definition of $\ffi_M$, we get $U_{\gam,p}(B)^{2p}<M+1$ a.s. in $\Om_a$, and the first part of the statement follows after applying Proposition \ref{prop:4.4}.

\smallskip

Let us also estimate the integral appearing in equation (\ref{eq:ell-u-tilde}): by~(\ref{y}), Lemma \ref{l1}, Theorem \ref{t1}, and the first part of the lemma, on $\Om_a$ we have
\beq
\left|\int_0^1 K(t,\xi) \sig(z_\xi)\, dB_\xi\right| \leq C \|K(t,\cdot)\|_\kappa (\|\sig'\|_\infty \|z\|_\kappa+\|\sig\|_\infty)\|B\|_\gam \leq C, \quad a.s. 
\label{eq:97}
\eeq
where the constant $C$ is positive, depends on $M,\sig,\kappa,\gam$ and indeed can be small enough whenever $\|\sig\|_\infty$ and $\|\sig'\|_\infty$ are small. Note that here we have used the fact that $\|K(t,\cdot)\|_\kappa$ $\leq C t^{1-\kappa}$, which can be easily deduced from the explicit expression of the kernel $K$.

\smallskip

On the other hand, still playing with equation (\ref{eq:ell-u-tilde}), 
\[
\tilde{G}_M(B) \left|\int_0^1 K(t,\xi) \sig(z_\xi)\, dB_\xi\right|\geq a \qquad a.s. \quad \text{on}\quad \Om_a.
\] 
Hence, (\ref{eq:97}) yields $\tilde{G}_M(B) \geq \frac{a}{C}$ almost surely on $\Om_a$, which concludes the proof.

\end{proof}

\subsection{Absolute continuity of the law}
With the previous changes in the equation we are considering, we are now ready to state and prove our result concerning the density of the law for $z_t$:

\begin{theorem}\label{thm:exis-den}
 Assume that $\sig$ satisfies the hypothesis of Theorem \ref{t2} and that $|\sig(y)|\geq \sig_0>0$ for all $y\in \re$, for some constant $\sig_0$. 
For any $t\in (0,1)$, we consider the random variable $z_t\in \D^{1,2}$ and $a>0$. Then, we have that
$\|\cd z_t\|_{\ch}>0$ a.s. on $\Om_a$. 

As a consequence, the law of $z_t$ restricted to $\R\setminus(-a,a)$ is absolutely continuous with respect to the Lebesgue measure.
\end{theorem}

\smallskip

Let us say a few words about the methodology we have followed in order to prove the result above: as in many instances, our density result will be obtained by bounding the Malliavin derivatives from below. Let us go back thus to equation (\ref{eq:malder}), which is the one satisfied by the Malliavin derivative $\cd z_t$. We wish to prove that a density exists for the random variable $z_t$ under a non-degeneracy condition of the form $\si(y)>\si_0$ for any $y\in\R$; we can assume, without loosing generality, that $\sig$ is positive. Our strategy will be based on the fact that $\cd z_t$ is a continuous function, and we will prove that, almost surely on $\Om_a$, the Malliavin derivative is negative on some non-trivial interval. This necessarily implies that the norm $\|\cd z_t\|_\hac$ cannot vanish. Let us however make the following observations:

\smallskip

\noindent
\textit{(i)} We will take advantage of the leading term $\Psi_s(t)$ in equation (\ref{eq:malder}) and we will analyze its increments. According to expression (\ref{eq:Psi-t}), these can only be assumed to be strictly negative when $s$ is small enough: we have not imposed any condition on $\tilde \mu_s$, and thus we can only rely on the upper bound (\ref{eq:bound-mu-del}), which is valid for $s$ close enough to 0. Let us insist again here on the fact that our change of cutoff in the elliptic equation we consider is meant to have  $\tilde \mu_s$ very small in a neighborhood of 0.

\smallskip

\noindent
\textit{(ii)}
The estimation of the integral part in equation (\ref{eq:malder}) involves some H\"older norms of the function $\xi\mapsto\cd_s z_\xi$. It is thus natural to think that the same should occur on the left hand side of this equation. Therefore, we are induced to consider increments of the form $\cd_{s}z_{t_2}-\cd_{s}z_{t_1}$ and perform our estimations on these quantities.

\smallskip

\noindent
\textit{(iii)} We shall tackle those increment estimates in a slightly more abstract setting, similar to Proposition \ref{prop2}: consider a function $(t,\eta)\mapsto w_t^{\eta}$, depending on two parameters $t,\eta\in\ou$. For $\eta\in\ou$, let $z^{\eta}$ be the solution to 
\begin{equation}\label{eq:z-eta-12}
z^{\eta}_{t}= w_t^{\eta} - \tilde G_M(B) \iou K(t,\xi) \, R_\xi \, z^{\eta}_{\xi} \, dB_\xi.
\end{equation}
In the equation above, $w$ and $R$ satisfy some suitable H\"older continuity assumption, and we assume the increments of $w$ to be also bounded from below. Notice that, for $\eta\le t$, the function $t\mapsto\cd_{\eta}z_t$ satisfies an equation of the form  (\ref{eq:z-eta-12}). Our aim is then to get an appropriate lower bound on the increments of $z^{\eta}$. This will be a consequence of the following lemma:

\begin{lemma}\label{lem:upper-bnd-w}
Let $\gam<H$ and $\kappa\in (0,1)$ be such that $\gam+\kappa>1$. For any $\eta\in [0,1]$, let $w^{\eta}$ be a function in $\cac^{\ka}$ satisfying the relation $|\delta w_{t_{1}t_{2}}^{\eta}|\le c_1 |t_2-t_1|\, \eta$ for any $\eta\le t_1\le t_2\le 1$ and $c_1<1$ small enough. 
Moreover, let $R\in \cac^{\ka}$ such that 
\beq
 \|R\|_\kappa \leq \frac{c_2}{M+1},
\label{eq:bound-R}
\eeq
for a small enough constant $c_2<1$ (see Proposition \ref{prop2}). 
Then the solution  $z^{\eta}$ to equation~(\ref{eq:z-eta-12}) is such that for all $\eta\le t_1\le t_2\le 1$,
\begin{equation}\label{eq:ineq-delta-z-eta}
|\delta z_{t_{1},t_{2}}^{\eta}| \le|t_2-t_1|\, \eta.
\end{equation}
If we further suppose that $\delta w_{t_{1}t_{2}}^{\eta} \le - c_1 |t_2-t_1|\, \eta$ for any $\eta\le t_1\le t_2\le 1$ and $c_1$ large enough, then we also get the bound
\begin{equation}\label{lemma9}
\delta z_{t_{1},t_{2}}^{\eta} \le -  c |t_2-t_1|\, \eta,
\end{equation}
for all $\eta\le t_1\le t_2\le 1$ and a small positive constant $c$.
\end{lemma}

\begin{proof}
Let us start by proving (\ref{eq:ineq-delta-z-eta}):
the solution $z^{\eta}$ to equation (\ref{eq:z-eta-12}) is obtained as the fixed point of an application $\Theta$ constructed as in the proof of Proposition \ref{prop2}. Namely, let us define the map $\Theta: \cac^\kappa\rightarrow \cac^\kappa$ by
\[
\Theta(y)_t := w_t^{\eta} - \tilde G_M(B) \int_0^t du \left( \int_u^1  R_\xi \, y_\xi dB_\xi\right)
+ t \, \tilde G_M(B) \int_0^1 \xi \, R_\xi \, y_\xi dB_\xi.
\]
Then under our standing assumptions, $z^{\eta}$ can be seen as the fixed point of the map $\Theta$. It is thus enough to check that, if $y$ verifies $|\delta y_{t_{1}t_{2}}|\le |t_2-t_1|\, \eta$ for all $\eta\le t_1\le t_2\le 1$, then $\hy:=\Theta(y)$ fulfills the same condition.

\smallskip

Let us write then
\begin{equation*}
\del\hy_{t_1t_2}=A_{t_1t_2}-  C_{t_1t_2}+  D_{t_1t_2},
\end{equation*}
with $A_{t_1t_2}=\del w_{t_1t_2}^{\eta}$ and 
$$
C_{t_1t_2}=\tilde G_M(B) \int_{t_1}^{t_2} du \left( \int_u^1  R_\xi \, y_\xi dB_\xi\right), \quad
D_{t_1t_2}=(t_2-t_1)\tilde G_M(B) \int_0^1 \xi \, R_\xi \, y_\xi dB_\xi.
$$
We shall bound those 3 terms separately for $\eta\le t_1\le t_2\le 1$.

\smallskip

$|A_{t_1t_2}|$ is bounded by assumption by $c_1 |t_2-t_1|\, \eta$. Furthermore, $|C_{t_1t_2}|$ is easily estimated as follows:
\begin{equation*}
|C_{t_1t_2}| \le 
\|R\|_{\ka} \,  \|y\|_{\cac^{\ka} ([t_1,1])}    M \, |t_2-t_1| \le
\|R\|_{\ka} \,  M \, |t_2-t_1| \, \eta,
\end{equation*}
thanks to our induction hypothesis. Hence, by (\ref{eq:bound-R}), we have 
\[
|C_{t_1,t_2}|\leq c_2 |t_2-t_1| \eta.
\]
Some similar considerations also yield $|D_{t_1,t_2}|\leq c_3 |t_2-t_1| \eta$ for a small enough constant $c_3$.  In order to complete the proof of  (\ref{eq:ineq-delta-z-eta}), it suffices thus to consider that $c_1, c_2$ small enough so that $c_1+c_2+c_3<1$.

\smallskip

Let us turn now to the proof of  (\ref{lemma9}): it is sufficient to go through the same computations as for (\ref{eq:ineq-delta-z-eta}) and take into account the lower bound on $\delta w_{t_{1},t_{2}}^{\eta}$. Details are left to the reader. We only notice that the constant $c_1$ has to be taken such that $c_1>c_2+c_3$, where the $c_2, c_3$ are the same constants of the proof of (\ref{eq:ineq-delta-z-eta}). 

\end{proof}

\smallskip

At this point, we already have the main tools in order to prove the main result of the section.


\begin{proof}[Proof of Theorem \ref{thm:exis-den}]
Taking into account that the Malliavin derivative $\cd z_t$ satisfies equation (\ref{eq:malder}), we will apply (\ref{lemma9}) to the following situation: $z^s_t=\cd_s z_t$, $R_\xi=\sig(z_\xi)$ and $w^s_t=\Psi_s(t)$, where we recall that
\[
\Psi_s(t)   =  \tilde{G}_M(B) \, \sig(z_s) \, K(t,s)
+ 2 \ffi'_M(U_{\gam,p}(B)^{2p}) \, \tilde{\mu}_s \, z_t
\]
and $\tilde \mu_s$ is defined by (\ref{eq:tilde-mu-rho}). We also remind that, throughout the proof, we have implicitly fixed $\om$ belonging to $\Om_a$.

First, note that the hypotheses on $\sig$ guarantee that (\ref{eq:bound-R}) is satisfied. Secondly, we observe that (\ref{lemma9}) is still true is we replace $\eta\le t_1\le t_2\le 1$ by $\eta\le t_1\le t_2\le T$, for any $T\in (0,1]$. In fact, we are going to apply that result for some small enough $T$.

 Let us prove that there exists $T$ such that $\del w^s_{t_1,t_2}\leq - c_1 |t_2-t_1| s$, for all $s\leq t_1\leq t_2\leq T$. We clearly have that
\beq
\label{eq:023}
\del w^s_{t_1,t_2}=\Psi_s(t_2)-\Psi_s(t_1)=-\tilde G_M(B) \sig(z_s) (t_2-t_1) s + 2\varphi_M'(U_{\gam,p}(B)^{2p}) \tilde \mu_s (\del z_{t_1,t_2}).
\eeq
By Lemma \ref{lemma:omega-a} and the non-degeneracy condition on $\sig$, the first term on the right-hand side of (\ref{eq:023}) can be bounded by $-c_4 (t_2-t_1) \eta$, where $c_4$ is some large enough constant (see the proof of Lemma \ref{lemma:omega-a}). We will check now that, for some small enough $T$, then
\begin{equation}\label{eq:bnd-phi'-mu}
2\varphi_M'(U_{\gam,p}(B)^{2p}) \tilde \mu_s (\del z_{t_1,t_2}) \leq c_5 (t_2-t_1) s,
\end{equation}
for some (small) constant $c_5$ (which may depend on $\om$). For this, we use the boundedness of $\varphi'_M$, apply 
Lemma  \ref{lemma:mu} (thus take $T$ small enough) and take into account the fact that, as it can be deduced from the existence result Theorem \ref{t1}, the solution $z$ is indeed Lipschitz continuous (with Lipschitz constant depending on $M$). Altogether this yields
\[
2\varphi_M'(U_{\gam,p}(B)^{2p}) \tilde \mu_s (\del z_{t_1,t_2}) \leq C \|B\|_{\gam+\ep}^{2p-1} \, s^\beta (t_2-t_1)
\leq c_5 (t_2-t_1) s,
\]
with $c_5=C \|B\|_{\gam+\ep}^{2p-1} \, T^{\beta-1}$, where we recall that $\beta=(2p-1)\ep-\gam$.

\smallskip

Therefore, taking $p$ large enough such that $c_5<c_4$ and plugging (\ref{eq:bnd-phi'-mu}) into (\ref{eq:023}), we obtain 
\[
\del w^s_{t_1,t_2}=\Psi_s(t_2)-\Psi_s(t_1) \leq - c_1 (t_2-t_1) s,\quad \text{for all} \quad s\leq t_1\leq t_2 \leq T,
\]
where $c_1$ can be large enough (since $c_4$ can be as well). 

Then, we are in position to apply (\ref{lemma9}) and we obtain that 
\[
\del(\cd z_t)_{t_1,t_2} \leq -|t_2-t_1| s, \quad \text{for all} \quad s\leq t_1\leq t_2 \leq T.
\]
This implies that we will be able to find $T_0<T$ such that
\beq
\del (\cd z_t)_{t_1,t_2} \leq -T_0 (t_2-t_1)<0, \quad \text{for all} \quad T_0\leq t_1\leq t_2 \leq T
\label{eq:024}
\eeq
and this holds almost surely in $\Om_a$. At this point, we have two possible situations:
\begin{itemize}
\item[(i)] If $\cd_{T_0} z_t \neq 0$, the continuity of the Malliavin derivative implies that it does not vanish in an interval around $T_0$. Thus the norm $\|\cd z_t\|_\hac$ must be strictly positive a.s. on $\Om_a$. 
\item[(ii)] If $\cd_{T_0} z_t =0$, condition (\ref{eq:024}) implies that $\cd_s z_t<0$ for $s\in (T_0,T]$, therefore we have again that $\|\cd z_t\|_\hac >0$ a.s. on $\Om_a$, by the continuity of $\cd z_t$.
\end{itemize}
This concludes the proof.  

\end{proof}



\begin{thebibliography}{99}

\bibitem{alos-nualart} Al\`os, E. and Nualart, D. Stochastic integration with respect to the fractional Brownian motion.  Stoch. Stoch. Rep. 75 (2003), no. 3, 129--152.

\bibitem{BH} Baudoin, F and Hairer, M. A version of H\"ormander's theorem for the fractional Brownian motion. Probab. Theory Related Fields 139 (2007), no. 3-4, 373--395.

\bibitem{bp} Buckdahn, R. and Pardoux, E. Monotonicity methods for white noise driven quasi-linear SPDEs.
Diffusion processes and related problems in analysis, Vol. I (Evanston, IL, 1989),  219-233, Progr. Probab., 22,
Birkhauser Boston, Boston, MA, 1990.


\bibitem{DGT} Deya, A., Gubinelli, M. and Tindel, S. Non-linear rough heat equations. Arxiv Preprint.

\bibitem{DP}
Donati-Martin, C. and Nualart, D. Markov property for elliptic stochastic partial differential equations.  Stochastics Stochastics Rep.  46  (1994),  no. 1-2, 107-115

\bibitem{err} Erraoui, M., Nualart, D. and Ouknine, Y. Hyperbolic stochastic partial differential equations with additive fractional Brownian sheet. Stochastics and Dynamics 3 (2003), 121--139.

\bibitem{Ev}
Evans, L. Partial differential equations. Second edition. Graduate Studies in Mathematics, 19. American Mathematical Society, Providence, RI, 2010.

\bibitem{Fe}Fernique, X. Fonctions al\'eatoires gaussiennes, vecteurs al\'eatoires gaussiens. Universit\'e de Montr\'eal, Centre de Recherches Math\'ematiques, Montr\'eal, QC, 1997.

\bibitem{Ga} Garsia, A. Continuity properties of Gaussian processes with multidimensional time parameter.  
Proceedings of the Sixth Berkeley Symposium on Mathematical Statistics and Probability Vol. II: Probability theory,  369--374. Univ. California Press (1972). 

\bibitem{GLT} Gubinelli, M., Lejay, A. and Tindel, S. Young integrals and SPDEs. Potential Anal. 25 (2006), no. 4, 307--326. 

\bibitem{GT} Gubinelli, M. and Tindel, S. Rough evolution equations. Ann. Probab. 38 (2010), no. 1, 1--75.

\bibitem{kusuoka} Kusuoka, K. The non-linear transformation of Gaussian measure on Banach space and its absolute continuity (I),
J. Fac. Sci. Univ. Tokyo IA 29 (1982) 567--597.

\bibitem{LT} Le\'on, J.A. and Tindel S.  Malliavin calculus for fractional delay equations. Arxiv Preprint (2009).

\bibitem{LS} Lototsky, S. V.; and Stemmann, K. Stochastic integrals and evolution equations with Gaussian random fields.  Appl. Math. Optim.  59  (2009),  no. 2, 203?232.

\bibitem{MS} Martinez, T. and Sanz-Sol\'e, M. A lattice scheme for stochastic partial differential equations of elliptic type in dimension $d\geq 4$. Appl. Math. Optim. 54 (2006), no. 3, 343--368. 

\bibitem{LQ} Lyons, T. and Qian, Z. System control and rough paths, Oxford University Press, 2002.

\bibitem{MN} Maslowski, B. and Nualart, D. Evolution equations driven by a fractional Brownian motion.
J. Funct. Anal. 202 (2003), No. 1, 277--305.

\bibitem{nualart} Nualart, D. The Malliavin calculus and related topics. Second edition. Probability and its Applications (New York). Springer-Verlag, Berlin, 2006.

\bibitem{NR} Nualart, D. and  Rascanu, A. Differential equations driven by fractional Brownian motion. 
Collect. Math. 53 (2002), No.1, 55--81.

\bibitem{NP}
Nualart, D. and  Pardoux, E. Second order stochastic differential equations with Dirichlet boundary conditions.  Stochastic Process. Appl.  39  (1991),  no. 1, 1-24. 

\bibitem{NS} Nualart, D. and Saussereau, B. Malliavin calculus for stochastic differential equations driven by a fractional Brownian motion. 
Stochastic Process. Appl. 119 (2009), no. 2, 391--409. 

\bibitem{Pr} Capietto, A. and Priola, E. Uniqueness in Law for Stochastic Boundary Value Problems. Journal of Dynamics and Differential Equations, to appear.

\bibitem{rv} Russo, F. and Vallois, P. Forward, backward and symmetric stochastic integration. Probab. Theory Related Fields 97 (1993), no. 3, 403--421.

\bibitem{rv-LNM} Russo, F. and Vallois, P. Elements of stochastic calculus via regularization. S\'eminaire de Probabilit\'es XL, 147--185, 
Lecture Notes in Math., 1899, Springer, Berlin, 2007.

\bibitem{QT} Quer-Sardanyons, L. and Tindel, S. The 1-d stochastic wave equation driven by a fractional Brownian sheet. 
Stochastic Process. Appl. 117 (2007), no. 10, 1448--1472.

\bibitem{ST} Sanz-Sol\'e, M. and Torrecilla, I. A fractional Poisson equation: existence, regularity and approximations of the solution. 
Stoch. Dyn. 9 (2009), no. 4, 519--548. 

\bibitem{St} Stroock, D. W. Probability Theory, 3rd Edition, Cambridge University Press, 
Cambridge, 1993.

\bibitem{Te} Teichmann, J. Another approach to some rough and stochastic partial differential equations. Arxiv Preprint. 

\bibitem{Ti}
Tindel, S. Diffusion approximation for elliptic stochastic differential equations.  Stochastic analysis and related topics, V (Silivri, 1994),  255-268, Progr. Probab., 38, Birkhäuser Boston, Boston, MA, 1996.

\bibitem{TTV} Tindel, S., Tudor, C. A. and  Viens, F. Stochastic evolution equations with fractional Brownian motion.
Probab. Theory Related Fields 127 (2003), no. 2, 186--204.

\bibitem{walsh} Walsh, J. B. An introduction to stochastic partial differential equations.
\'Ecole d'\'et\'e de probabilit\'es de Saint-Flour XIV - 1984, Lect. Notes Math. 1180, 265-437 (1986).

\bibitem{young} Young, L. C. An inequality of H\"older type, connected with Stieltjes integration. Acta Math. 67 (1936), 251--282.


\end{thebibliography}
\end{document}